\documentclass{article}
\usepackage{graphicx} 
\usepackage{titlesec}
\usepackage{mathtools}
\usepackage{blkarray} 
\usepackage{multirow}
\usepackage{fullpage}
\usepackage{url}

\usepackage{amsmath,graphicx,color,amsfonts,subfigure}
\usepackage[noend]{algorithmic}
\usepackage{version,xspace}
\usepackage{tikz}
\usepackage{centernot}
\usepackage{dcolumn}
\usetikzlibrary{calc}

\def\vspecdots{\vbox{\baselineskip=2pt \lineskiplimit=0pt 
    \kern2pt \hbox{.}\hbox{.}\hbox{.}}} 
\newcommand{\until}[1]{\{1,\dots, #1\}}

\newcommand{\subscr}[2]{#1_{\textup{#2}}}

\newcommand{\setdef}[2]{\{#1 \; | \; #2\}}

\newcommand{\union}{\operatorname{\cup}}

\newcommand{\diag}[1]{\ensuremath{\operatorname{diag}[#1]}}
\newcommand{\Vect}[1]{\ensuremath{\operatorname{vec}(#1)}}

\newcommand{\kron}{\operatorname{\otimes}}
\newcommand{\refp}[1]{(\ref{#1})}

\newcommand{\vectorMat}{\ensuremath{\operatorname{vec}}}

\newcommand{\PP}{\mathcal{P}}
\newcommand{\fmt}{m}
\newcommand{\FMT}{M}
\newcommand{\MFMT}{\mathcal{M}}
\newcommand{\multEvaders}{P_\textnormal{e}^{(1)},P_\textnormal{e}^{(2)},\dots,P_\textnormal{e}^{(M)}}
\newcommand{\multEvadersCTMC}{Q_\textnormal{e}^{(1)},Q_\textnormal{e}^{(2)},\dots,Q_\textnormal{e}^{(M)}}
\newcommand{\multPursuers}{P_\textnormal{p}^{(1)},P_\textnormal{p}^{(2)},\dots,P_\textnormal{p}^{(L)}}
\newcommand{\multPursuersCTMC}{Q_\textnormal{p}^{(1)},Q_\textnormal{p}^{(2)},\dots,Q_\textnormal{p}^{(L)}}
\newcommand{\pursuersKronEvaders}{(P_\textnormal{p}^{(1)} \kron P_\textnormal{p}^{(2)} \kron \dots \kron P_\textnormal{p}^{(L)} \kron P_\textnormal{e}^{(1)} \kron P_\textnormal{e}^{(2)} \kron \dots \kron P_\textnormal{e}^{(M)})}

\newcommand\oprocendsymbol{\hbox{$\square$}}
\newcommand\oprocend{\relax\ifmmode\else\unskip\hfill\fi\oprocendsymbol}

\def \bs {\boldsymbol}
\def \mc {\mathcal}

\newcommand{\G}{\mc G}
\newcommand{\V}{V}

\newcolumntype{d}[1]{D{.}{.}{#1}}

\DeclareSymbolFont{bbold}{U}{bbold}{m}{n}
\DeclareSymbolFontAlphabet{\mathbbold}{bbold}
\newcommand{\vect}[1]{\mathbbold{#1}}
\newcommand{\vectorones}[1][]{\vect{1}_{#1}}
\newcommand{\vectorzeros}[1][]{\vect{0}_{#1}}

\newtheorem{theorem}{Theorem}
\newtheorem{corollary}{Corollary}
\newtheorem{lemma}{Lemma}
\newtheorem{remark}{Remark}

\newenvironment{proof}[1][Proof]{\begin{trivlist}
\item[\hskip \labelsep {\bfseries #1}]}{\end{trivlist}}
\renewcommand{\theenumi}{(\roman{enumi}}

\title{The Meeting Time of Multiple Random Walks
  \thanks{This work has been supported by Air Force Office of Scientific Research award FA9550-15-1-0138.}}

\author{Mishel George \thanks{Mechanical Engineering, University of California, Santa Barbara, CA 93106-5070~(mishel@engineering.ucsb.edu).} \and Rushabh Patel \thanks{Northrop Grumman Aerospace Systems, Redondo Beach, CA~(rushabh.patel@ngc.com).} \and Francesco Bullo \thanks{Center for Control, Dynamical Systems and Computation, University of California, Santa Barbara,
CA 93106-5070 (bullo@engineering.ucsb.edu).}}

\begin{document}
\graphicspath{{./figures/}}
\maketitle                    
\renewcommand{\theenumi}{(\roman{enumi})}%

\begin{abstract}
  This article rigorously analyzes the \emph{meeting time} between
pursuers and evaders performing random walks on digraphs. There exist
several bounds on the expected meeting time between random walkers on
graphs in the literature, however, closed-form expressions are limited
in scope. By utilizing the notion that multiple random walks on a
common graph can be understood as a single random walk on the
Kronecker product graph, we are able to provide the first analytic
expression for the meeting time in terms of the transition matrices of
the random walkers when modeled by either discrete-time Markov chains
or continuous-time Markov chains.  We further extend the results to
the case of multiple pursuers and multiple evaders performing
independent random walks. We present various sufficient conditions for
pairs (or tuples) of transition matrices that satisfy certain
conditions on the absorbing classes for which finite meeting times are
guaranteed to exist.
\end{abstract}

\section{Introduction}\label{sec:Intro}
\subsection{Problem description and motivation}
In this paper, we examine the meeting time between two
groups of random walkers. This problem is motivated by a group of
pursuers trying to intercept a group of evaders. The meeting time,
in the context of this paper, describes the average time till a
first encounter occurs between one of the pursuers and one of the evaders
given initial positions of the pursuers and the evaders. This notion of two
adversarial mobile groups wherein one of the groups is trying to
intercept members of the other group appears under several names:
pursuit-evasion games~\cite{RV-OS-HJK-DHS-SS:02}, predator-prey
interactions~\cite{CC-AF-TR:09}, cops and robbers
games~\cite{AB-PG-GH-JK:09} and princess-monster
games~\cite{SA-RF-RL-GO:08}. Our primary motivation is the design of
stochastic surveillance strategies for quickest detection of mobile
intruders. Single and multi-agent surveillance strategies appear in
environmental monitoring~\cite{YE-AS-GAK:08,FP-AF-FB:09v}, minimizing
emergency vehicle response times~\cite{THB-JSK:02}, traffic routing
and border patrol~\cite{MP-CD-VD-AK-YW:03,SS-SM-FB:06f}. More broadly
random walks on networks appear in many areas of research: they are
used to describe effective resistance in electrical
networks~\cite{PGD-JLS:84,PT:91}, for link-prediction and information
propagation in social networks~\cite{LB-JL:11,CG-MM-AS:04}, and in designing
search algorithms on networks~\cite{QL-PC-EC-KL-SS:02, MSS-DT-SB:14}. Aside from our proposed application to stochastic surveillance, the meeting
time has direct applications to information flow in distributed
networks~\cite{SRE-TB:16}, self-stabilization of tokens~\cite{AI-MJ:90}
and measuring similarity of objects~\cite{GJ-JW:02}.
\begin{figure}
  \centering
  \includegraphics[width = 0.5\textwidth]{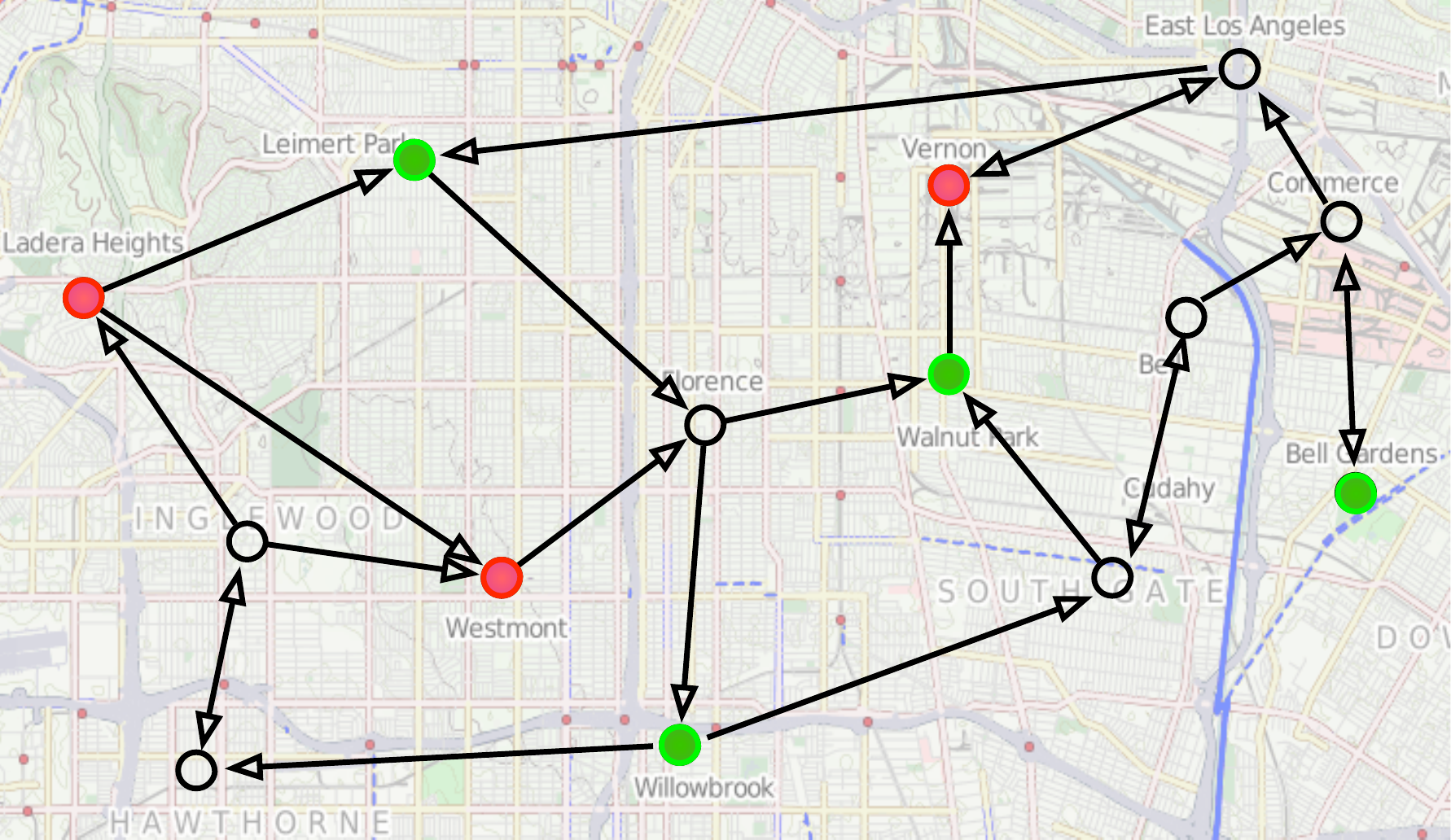}
  \caption{Multiple pursuers (green) and multiple evaders (red) performing random walks on a digraph.}
\end{figure}
\subsection{Literature review}
Early interest in meeting times was motivated by applications to self-stabilizing token management schemes~\cite{PT-PW:91}. In a token management scheme, only one of the many processors on a distributed network is enabled to change state or perform a particular task, and this processor is said to possess the token. If two tokens meet then they collapse into a single token. Israeli and Jalfon suggest a scheme in which the token is passed randomly to a neighbor~\cite{AI-MJ:90}. In a general connected, undirected, n-vertex graph they were able to obtain an exponential bound for the meeting time of two tokens in terms of the maximum degree and the diameter of the graph. Coppersmith et al~\cite{DC-PR-PW:93} improved the bound to be polynomial in the number of nodes by bounding the meeting time in terms of the pairwise hitting time from the starting nodes of the tokens to hidden vertices. In~\cite{DC-PR-PW:93, AI-MJ:90, PT-PW:91} the notion of the meeting time involves the tokens being moved asynchronously by an adversary whose objective is to maximize meeting time by playing only one of the two tokens. Bshouty et al~\cite{NHB-LH-JWG:99} obtain a bound on the meeting time of several such tokens in terms of the meeting time of two tokens. Bounds for meeting times of two identical independent continuous-time reversible Markov chains in terms of the pairwise hitting times of the chain are mentioned in the work by Aldous~\cite{DJA:91}. Several variations of ``cat-mouse'' games are discussed in~\cite{DA-JAF:02} wherein bounds are obtained in terms of the pairwise hitting time or the variation-threshold time (a measure of rate of convergence to stationary distributions) depending on the Markov chains being discrete-time or continuous-time.

Several other metrics have been used to describe single and multiple random walks of graphs. Here we describe three of the most relevant quantities. A closely related metric is the hitting time which is the time taken by a single random walker to travel between nodes of a graph. The hitting time of a finite irreducible Markov chain first appeared in~\cite{JGK-JLS:76}, however, it was rediscovered for finite reversible Markov chains in~\cite{AZB-ARK:89}. Several bounds have been obtained for the hitting time for various graph topologies~\cite{JJH:13,ML-GL:02}. Many closed-form formulations exist to compute the hitting time for a single random walker~\cite{JJH:13,SK:10,RP-PA-FB:14b}. The authors in~\cite{RP-AC-FB:14k} obtain a closed-form solution for the hitting time of multiple random walkers. Another related notion is that of cover times which is the expected time taken by random walkers to hit every node on the graph~\cite{AZB-ARK:89, RE-TS:11}. There are several works relating the cover time to the hitting times of a Markov chain~\cite{YN-HO-KS-MY:10}; many of these works bound the cover time in terms of worst-case pairwise hitting times. Finally, the coalescence time of multiple random walkers is a quantity which is widely studied especially in the context of voter models~\cite{CC-RE-HO-TR:13,JTC:89}. Two random walks coalesce into one when they share the same node. Bounds for the coalescence time in terms of the worst case pairwise hitting times are discussed in~\cite{DA-JAF:02}. More recently, Cooper et al bounded the coalescence time using the second largest eigenvalue of the transition matrix~\cite{CC-RE-HO-TR:13}.

Stochastic vehicle routing strategies have the desirable property that an intruder can not predictably plan a path to avoid surveillance agents. The authors in~\cite{JG-JB:05,KS-DMS-MWS:09} use Markov chain Monte Carlo methods to design surveillance strategies. Minimizing the mean hiting time by introducing a novel convex optimization formulation is used to design strategies in~\cite{RP-PA-FB:14b}. The notion of group hitting time for multiple random walkers is used in optimizing transition matrices for multiple agents in~\cite{RP-AC-FB:14k, AC-RP-FB:14k-conference}. In~\cite{PA-FB:15e} the mean hitting time in conjunction with multiple parallel CUSUM algorithms at various nodes of interest in the graph are used to describe a policy which ensures quickest average time to detection of anomalies. In the strategies mentioned in these works the intruder/anomaly is assumed to be stationary. The policies for surveillance derived in this paper are for mobile intruders modeled by Markov chains.

\subsection{Contributions}
Given the above, there are several contributions in this paper. First, we provide a set of necessary and sufficient conditions which characterize when the meeting times between a single pursuer and a single evader is finite for arbitrary Markov chains. To the best of our knowledge the bounds in the literature were obtained for meeting times between ergodic Markov chains where the meeting times are guaranteed to be finite. However we extend the notion to generic transition matrices as opposed to equal neighbor models which are studied in many of the works mentioned above, and we discuss at length when the meeting times are finite based on the existence of walks of equal length to common nodes. Second, we provide a closed-form solution to the meeting time of two independent Markov chains by utilizing the Kronecker product of the transition matrices. Both these results are obtained using a technical approach which takes advantage of the properties of Kronecker products of graphs. We further use this closed-form expression to perform comparisons with existing bounds in the literature. Indeed we see that the bounds are very conservative for most graphs. Third, we provide a set of sufficient conditions in terms of the absorbing classes of the pursuer and evader chains which guarantee finite meeting times. Fourth, we extend the treatment to multiple pursuers and multiple evaders. Finally, we obtain conditions for the meeting times between two continuous-time Markov chains to be finite and provide closed-form results for this case, and further extend it to multiple pursuer and evader groups when dictated by multiple transition rate matrices.

To the best of our knowledge, this paper provides the first closed-form solutions for the computation of the meeting time between two Markov chains for both discrete-time and continuous-time time indices. Two closely related references are as follows: first, a system of equations for computing meeting times for independent identical random walks on graphs with irreducible transition matrices, where the transition matrices are limited to equal-neighbor weights, were obtained using Laplace transform techniques in~\cite{TO:15}. Second, Kronecker products and vectorization techniques have been used to compute the Simrank of information networks which has interpretations in terms of meeting times~\cite{CL-JH-GH-XJ-YS-YY-TW:10}. Our work is different in several ways. First, we consider absolutely generic transtion matrices which need not be identical. Second, we present expressions here which are valid for reducible transition matrices. Third, we present meeting time expressions for the case of multiple pursuers which would correspond to multiple infecting particles in~\cite{TO:15}. Finally, we provide insight into when meeting times are finite by connecting this notion to the existence of walks on the Kronecker graph.

\subsection{Organization}
This paper is organized as follows. In Section~\ref{sec:Notation} we introduce notation that is used throughout the paper and review useful concepts. In Section~\ref{sec:Single_pursuer_evader} we introduce our formulation for the meeting times of pairs of Markov chains, and also define sets of pairs of matrices for which finite meeting times exist. In Section~\ref{sec:multiple_pursuers_evaders} we extend the notion of the meeting time to multiple pursuers and evaders. In Section~\ref{sec:CTMC_meeting} we obtain closed-form expressions for continuous-time Markov chains. Finally, in Section~\ref{sec:conclusion} we present conclusions.

\section{Notation} \label{sec:Notation}
In this section we define various useful concepts and notation. We provide an 
overview of some facts and results on Markov chains and introduce notation that will be used throughout the paper to deal with vectors and matrices, the Kronecker product, and discuss Markov chains on graphs.
\subsection{Markov chains} \label{subsec:MarkChains}
A \emph{Markov chain} is a sequence of random variables taking value in the \emph{finite} set $\until{n}$ with the Markov property, namely that the future state depends only on the present state.

Let $X_k \in \until{n}$ denote the location of a random walker at time $k \in \{0,1,$ $2,\dots\}$, then a discrete-time Markov chain is \emph{time-homogeneous} if $\mathbb{P}[X_{n+1}=j\,|\,X_n =  i]=\mathbb{P}[X_{n}=j\,|\,X_{n-1} = i]=p_{i,j}$, where $P =[p_{i,j}] \in \mathbb{R}^{n  \times n}$ is the \emph{transition matrix} of the Markov chain. By definition, each transition matrix $P$ is row-stochastic,  i.e., $P \vectorones[n] = \vectorones[n]$. The period of a state is defined as the greatest common divisor of all $t$ such that $\{t\geq 1\,|\,\mathbb{P}[X_t=i\,|\,X_0=i]\neq 0\}$. A state whose period is one is referred to as an \emph{aperiodic} state. It can be shown that in a \emph{communicating class}(defined below) all states share the same period. For more details on discrete-time Markov chains refer \cite[Chapter 8]{CDM:01}.

Let $X_t \in \until{n}$ denote the location of a random walker at time $t\in\mathbb{R}^{+}$, then a continuous-time Markov chain is \emph{time-homogeneous} if $\mathbb{P}[X_{t^{'}+t}=j\,|\,X_{t^{'}} =  i]=p^t_{i,j}$ for all $t \geq 0, t^{'} \geq 0$, where $P(t) = p_{i,j}^{t} \in \mathbb{R}^{n  \times n}$ is the transition matrix of the Markov chain. The evolution of the continuous-time Markov chain is determined by the solution to the first-order differential equation $P^{'}(t) = P(t) Q$, where $P(t)=p^{t}_{i,j}$ and $Q$ is a \emph{transition rate matrix} which satisfies $Q\vectorones[n]=\vectorzeros[n]$. For more details on continuous-time Markov chains refer \cite[Chapters 2 \& 3]{JRN:97}. A continuous-time Markov chain is said to be \emph{ergodic} if it is irreducible.

Consider two states $i$ and $j$ belonging to a Markov chain. We say $i$ \emph{communicates} with $j$ if  $p^t_{i,j} \neq 0$ for some $t>0$. For a subset of states $X \subset \until{n}$, we say that $X$ forms a \emph{communicating class} if for every state $i,j \in X$ the states communicate with each other, i.e $\mathbb{P}[X_{t}=j\,|\,X_0=i]\neq 0$ and $\mathbb{P}[X_{t^{'}}=i\,|\,X_0=j]\neq 0$ for some $t,t'\geq0$. An \emph{absorbing class} $A$ of a Markov chain is a communicating class such that the probability of escaping the set is zero, i.e $p^t_{i,j} = 0$ for all $t>0$ for all $i \in A, j\notin A$. If a communicating class is not absorbing, then it is called a \emph{transient class}. In general, a Markov chain will have multiple absorbing and transient classes. If a Markov chain has only a single absorbing class then it is referred to as a \emph{single absorbing Markov chain}. 

If a Markov chain is single absorbing, then a unique stationary distribution $\bs \pi$ exists. The vector $\bs \pi \in \mathbb R ^{n \times 1}$ is a \emph{stationary distribution} of a discrete-time Markov chain with transition matrix $P$ if $\sum_{i=1}^n \bs \pi_i = 1$ and $\bs \pi^\top P = \bs \pi^\top$ and of a continuous-time Markov chain with transition rate matrix $Q$ if $\sum_{i=1}^n \bs \pi_i = 1$ and $\bs \pi^\top Q =0$. A Markov chain is \emph{irreducible} if the absorbing class is the entire set of states $\until{n}$. A discrete-time Markov chain is said to be \emph{ergodic} if it is irreducible and aperiodic.

\subsection{Matrix notation} \label{subsec:KronProps}
We use the notation $A = [a_{i_1\dots i_l,j_1\dots j_m}]$ to denote the matrix generated by elements $a_{i_1\dots i_l,j_1\dots j_m}$, where the rows of $A$ are determined by cycling through indices $i_{l}$ followed by $i_{l-1}$ and so on until $i_1$, and the columns of $A$ are determined by cycling through indices $j_{m}$ followed by $j_{m-1}$ and so on until $j_1$.  For example, consider $i_1,i_2,j_1,j_2  \in \until{n}$, then 

\begin{align*} 
  A  =[a_{i_1 i_2, j_1 j_2}] = \left[
  \begin{array}{cccccccc}
    a_{11,11} & a_{11,12}  & \dots & a_{11,1n} & a_{11,21} &\dots & a_{11,nn} \\ 
    a_{12,11} & a_{12,12}  & \dots & a_{12,1n} & a_{12,21} &\dots & a_{12,nn} \\ 
    \vdots & \vdots & \dots & \dots & \dots & \vdots & \vdots \\
    a_{1n,11} & a_{1n,12} & \dots & a_{1n,1n} & a_{1n,21} &\dots & a_{1n,nn} \\
    a_{21,11} & a_{21,12} & \dots & a_{21,1n} & a_{21,21} &\dots & a_{21,nn} \\
    \vdots & \vdots & \dots & \dots & \dots & \vdots & \vdots \\
    a_{nn,11} & a_{nn,12} & \dots & a_{nn,1n} & a_{nn,21} & \dots & a_{nn,nn}
  \end{array} \right].
\end{align*}
For the case where $A = [a_{i,j}]$ this corresponds to the classic interpretation 
with element $a_{i,j}$ in the $i$-th row and $j$-th column of $A$. We use the notation $\diag{a}$ to indicate the diagonal matrix generated by vector $\bs a$ and 
$\Vect{A}$ to indicate the vectorization of a matrix $A \in \mathbb{R}^{n 
  \times m}$ where $\Vect{A} = [A(1,1)$,$\dots$,$A(n,1)$,$A(1,2)$,$\dots$,\newline$A(n,2)$,
$\dots$,$A(1,m)$ ,$\dots,A(n,m)]^\top$. In other words, even if we define $A$ as $A =  [a_{i_1i_2,j_1j_2}]$, the vector $\Vect{A} = \Vect{[a_{i_1i_2,j_1j_2}]}$ is simply a stacking of the columns of $A$.
 
Let $I_n \in \mathbb{R}^{n\times n}$ denote the identity matrix of size $n$, $\vectorones[n]\in \mathbb{R}^{n\times 1}$ denote the vector of ones of size $n$, and $\vect{e}_1,\vect{e}_2,\dots,\vect{e}_n \in \mathbb{R}^{n\times 1} $ denote vectors with unity in the row indicated by the subscript. We define a generalized Kronecker delta function $\delta_{i_1i_2\dots i_l,j_1j_2\dots j_m}$, by
\begin{equation*}
  \delta_{i_1\dots i_l,j_1j_2\dots j_m} = 
  \begin{cases}
    1, & \text{if }\exists~l',m' \text{ such that } i_{l'} = j_{m'}\text{ for any }1\leq l' \leq l, 1\leq m' \leq m,\\ 
    0, & \text{otherwise}.
  \end{cases}
\end{equation*}
We use the subscript p,~e or superscript $(\textnormal{p}),(\textnormal{e})$ to delineate between quantities associated with pursuers and evaders.

We are now ready to review some useful facts about Kronecker products. The Kronecker product, represented by the symbol $\kron$, of two matrices $A \in \mathbb{R}^{n \times m}$ and $B \in \mathbb{R}^{q \times r}$ is an $nq \times mr$ matrix given by
\begin{align*} 
  A\kron B = \left[
  \begin{array}{ccc}
    a_{1,1} B & \dots & a_{1,m} B \\ 
    \vdots& \ddots  &\vdots \\
    a_{n,1} B &  \ddots& a_{n,m}B 
  \end{array} \right].
\end{align*}
The Kronecker product is bilinear and has many useful properties,
two of which are summarized in the following Lemma; see~\cite[Chapter 4]{RAH-CRJ:94}
for more information.
\begin{lemma}[Properties of the Kronecker product] \label{def:KronProps1}
  Given the matrices $A,B,C$ and $D$, the following relations hold for the Kronecker product.
  \begin{enumerate}
  \item $(A \kron B) (C \kron D) = (A  C) \kron (B  D)$,
  \item $(B^\top \kron A) \Vect{C}= \Vect{ACB}$,
  \end{enumerate}
  where it is assumed that the matrices are of appropriate dimension when
  matrix multiplication or addition occurs.
\end{lemma}
\subsection{Markov chains on graphs}
In this paper, for discrete-time Markov chains we consider weighted digraphs $\G = (\V,E,P)$ with node sets $V :=\until{n}$, edge set $E \subset \V \times \V$, and associated transition matrix $P=[p_{i,j}]$ with the property that $p_{i,j}\geq0$ if $(i,j) \in E$ and $p_{i,j}=0$ otherwise. The weight of the edge $(i,j)$ is interpreted as the weight associated with the probability of transition from node $i$ to node $j$. The nodes of the graph are equivalent to the states of the Markov chain. We say there exists a \emph{walk} of \emph{length} $\ell$ from node $i_1$ to node $i_l$ if there exists a sequence of nodes $i_2,\dots,i_{\ell-2}$ such that $p_{i_k,i_{k+1}}>0$ for $1\leq k \leq \ell-1$.

In this paper, for continuous-time Markov chains we consider weighted digraphs $\G = (\V,E,Q)$ with node sets $\V :=\until{n}$, edge set $E \subset \V \times \V$, and associated transition rate matrix $Q=[Q_{i,j}]$ with the property that $q_{i,j}\geq0$ if $(i,j) \in E$, $q_{i,j}=0$ otherwise and $q_{i,i} = -\sum_{j\in V}q_{i,j}$. The weight of the edge $(i,j)$ is interpreted as the rate of transition from node $i$ to node $j$. One could also look at the entry $-1/q_{i,i}$ as the average time at which the walker leaves node $i$ and $1/q_{i,j}$ as the average time for a jump from $i$ to $j$. We say there exists a walk from node $i_1$ to node $i_l$ if there exists a sequence of nodes $i_2,\dots,i_{\ell-1}$ such that $q_{i_l,i_{l+1}}>0$ for $1\leq l \leq \ell-1$.

\section{Single pursuer and single evader} \label{sec:Single_pursuer_evader}
In this section, we formulate the meeting time between two discrete-time Markov chains. We provide necessary and sufficient conditions for the finiteness of the meeting times given any initial starting positions on the graph. We specify certain sets of pairs of transition matrices where finite meeting times are guaranteed to exist and discuss when it is appropriate to define the notion of a mean meeting time. Finally, we compare the exact values obtained with this expression to some of the bounds from the literature on meeting times.
\subsection{The meeting time of two Markov chains}
Consider the pursuer and evader performing random walks on a set of nodes $\V :=\until{n}$  with digraphs $\G_\textnormal{p} =(\V,E_\textnormal{p},P_\textnormal{p})$, $\G_\textnormal{e}=(\V,E_\textnormal{e},P_\textnormal{e})$, edge sets $E_\textnormal{p},E_\textnormal{e}\subset \V \times \V$, and transition matrices $P_\textnormal{p}$, $P_\textnormal{e}$. The matrix $P_\textnormal{p}$ satisfies $p_{i,j}^{(\textnormal{p})} \geq 0$ if $(i,j) \in E_\textnormal{p}$ and $p_{i,j}^{(\textnormal{p})} = 0$ if $(i,j) \notin E_\textnormal{p}$. Similarly $P_\textnormal{e}$ satisfies similar properties to be a well-defined transition matrix on $\G_\textnormal{e}$.  

Let $X_t^{(\textnormal{p})},X_t^{(\textnormal{e})} \in\until{n}$  be the location of the two agents at time $t\in\{0,1,2,\dots\}$.
\newline
For any two start nodes $i,j$, the \emph{first meeting time from $i$ and $j$}, denoted by  $T_{i,j}$, is the first time that both random walkers meet when starting from nodes $i$ and $j$, respectively. More formally,
\begin{equation*}
  T_{i,j} =\min \setdef{t\geq 1}{X_t^{(\textnormal{p})} = X_t^{(\textnormal{e})}\text{ given that } X_0^{(\textnormal{p})} = i \text{ and } X_0^{(\textnormal{e})}=j}.
\end{equation*}
Note that the first meeting time can be infinite. It is easy to
construct examples in which the two agents never meet. Let
$\fmt_{i,j} = \mathbb{E}[T_{i,j}]$ be the expected first
meeting time starting from nodes $i$ and $j$. For the sake of brevity, we shall refer to the expected first meeting time as just the meeting time.

\begin{theorem}[The meeting time of two Markov chains]\label{thm:necessary_sufficient_finite}
  Consider two Markov chains with transition matrices $P_\textnormal{p}$ and
  $P_\textnormal{e}$ defined on a digraph $\G$ with nodeset $V = \until{n}$.  The
  following statements are equivalent:
  \begin{enumerate}
  \item for each pair of nodes $i,j$, the meeting time
    $\fmt_{i,j}$ from nodes $i$ and $j$ is finite,
  \item for each pair of nodes $i,j$, there exists a node $j$ and
    a length $\ell$ such that a walk of length $\ell$ exists
    from $i$ to $k$ and a walk of length $\ell$ exists from
    $j$ to $k$,
  \item for each pair of nodes $i,j$, there exists a
    walk in the digraph associated with the stochastic matrix \textnormal{$P_\textnormal{p} \kron
    P_\textnormal{e}$} from $(i,j)$ to a node $(k,k)$, for some
    $k\in\until{n}$, and

  \item the sub-stochastic matrix \textnormal{$(P_\textnormal{p} \kron P_\textnormal{e})E$}
    is convergent and the vector of meeting times is given by
    \textnormal{
    \begin{align} \label{eq:Vector2agent}
      \Vect{\FMT} = (I_{n^2}-(P_\textnormal{p}\kron P_\textnormal{e})E)^{-1} \vectorones[n^2], 
    \end{align}}
    where $\FMT\in \mathbb{R}^{n \times n}$ and $E \in \mathbb{R}^{n^2 \times n^2}$ is a binary diagonal matrix with diagonal entries $\vectorones[n^2] -\Vect{I_n}$.
  \end{enumerate}
\end{theorem}

\begin{proof}
  For the nodes $i$ and $j$, the first meeting time satisfies the recursive formula
  \begin{equation*}
    T_{i,j} = 
    \begin{cases}
      1, & \text{w.p.}~\sum_k p_{i,k}^{(\textnormal{p})}p_{j,k}^{(\textnormal{e})},\\ 
      T_{k_1,h_1}+1, & \text{w.p.}~ p_{i,k_1}^{(\textnormal{p})}p_{j,h_1}^{(\textnormal{e})}, k_1\neq h_1.
    \end{cases}
  \end{equation*}
  Taking the expectation we have
  \begin{equation*}
    \begin{split}
      \mathbb{E}[T_{i,j}] &=\sum_{k} p_{i,k}^{(\textnormal{p})}p_{j,k}^{(\textnormal{e})} + \sum_{k_1 \neq h_1 } p_{i,k_1}^{(\textnormal{p})}p_{j,h_1}^{(\textnormal{e})} (\mathbb{E}[T_{k_1,h_1}]+1),\\
      &= \sum_{k_1}\sum_{h_1}p_{i,k_1}^{(\textnormal{p})} p_{j,h_1}^{(\textnormal{e})} + \sum_{k_1 \neq h_1 } p_{i,k_1}^{(\textnormal{p})}p_{j,h_1}^{(\textnormal{e})} \mathbb{E}[T_{k_1,h_1}], \\
      &=1 + \sum_{k_1 \neq h_1 } p_{i,k_1}^{(\textnormal{p})}p_{j,h_1}^{(\textnormal{e})} \mathbb{E}[T_{k_1,h_1}].
    \end{split}
  \end{equation*}
  Let $\fmt_{i,j} = \mathbb{E}[T_{i,j}]$ for every $i,j \in \until{n}$ and let $\FMT=[\fmt_{i,j}]$. Note that the entries of $\FMT$ can be written as 
  \begin{equation*}
    \begin{split}
      \fmt_{i,j} &= 1 + \sum_{k_1\neq h_1}  p_{i,k_1}^{(\textnormal{p})} p_{j,h_1}^{(\textnormal{e})} \fmt_{k_1,h_1}, \\
      \implies \fmt_{i,j} &= 1 + \sum_{k_1=1}^n p_{i,k_1}^{(\textnormal{p})} \sum_{h_1=1}^n  \fmt_{k_1,h_1} p_{j,h_1}^{(\textnormal{e})}  - \sum_{k=1}^n p^{(\textnormal{p})}_{i,k}p^{(\textnormal{e})}_{j,k}\fmt_{k,k},\\
      \implies \FMT &= \vectorones[n] \vectorones[n]^\top + P_\textnormal{p}(\FMT-\FMT_d)P_\textnormal{e}^\top,
    \end{split}
  \end{equation*}
  where $\FMT_d \in \mathbb{R}^{n\times n}$ is a diagonal matrix with only the diagonal elements of $\FMT$ . We have used the property that $(ABC)_{i,j} = \sum_kA_{i,k}\sum_lB_{k,l}C_{l,j}$ to obtain the equation in matrix form. Rewriting the equation in vector form and using Lemma~\ref{def:KronProps1} gives
  \begin{equation*}
    \begin{split}
      \Vect{\FMT} &= \vectorones[n^2] + (P_\textnormal{p}\kron P_\textnormal{e})(\Vect{\FMT}  -\Vect{\FMT_d}),\\
      \Vect{\FMT} &= \vectorones[n^2] + (P_\textnormal{p}\kron P_\textnormal{e}) (I_{n^2}-\Vect{I_n}) \Vect{\FMT}\\ 
      \Vect{\FMT} &= \vectorones[n^2] + (P_\textnormal{p}\kron P_\textnormal{e}) E \Vect{\FMT}
    \end{split}
  \end{equation*}
  If the matrix $I_{n^2}-(P_\textnormal{p}\kron P_\textnormal{e})E$ is invertible then we have a unique solution to the meeting times. We shall now show that the finiteness of meeting times as in $(i)$ is equivalent to the existence of walks of equal length to common nodes as mentioned in $(ii)$ and in $(iii)$, which guarantees invertibility of $I_{n^2}-(P_\textnormal{p}\kron P_\textnormal{e})E$ in $(iv)$.

\begin{figure}[t]
  \includegraphics[width = 0.99\textwidth]{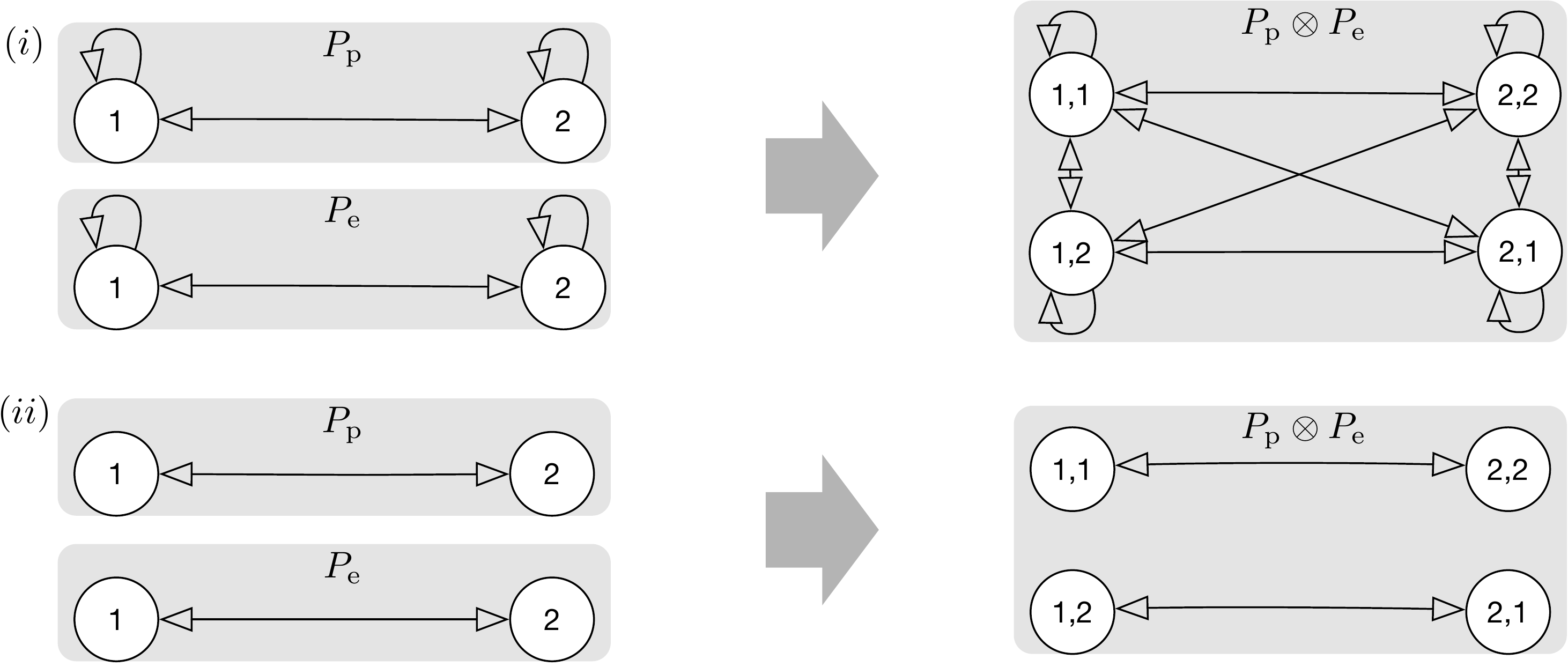}
  \caption{The pursuer-evader pair in $(i)$ has finite meeting times as every node has a walk to the common nodes $(1,1)$ and $(2,2)$ in the Kronecker graph. However, in $(ii)$ there exists no walks to common nodes from $(1,2)$ and $(2,1)$.}
\end{figure}

  We start by proving that $(i) \implies (ii)$. If we assume that $(i) \centernot \implies (ii)$, then there exists a pair of nodes $i$ and $j$ such that the meeting time is finite and there exists no walk of equal length to any node in $V$. However if there exists no walk of equal length to a common node, then the agents never meet and the meeting time is always infinite. Hence by contradiction $(i) \implies (ii)$.

  Next, we show that $(ii) \iff (iii)$. The Kronecker product of the transition matrices gives a joint transition matrix for the agents over the set of nodes $V \times V$. The $(i,j)$ entry of the matrix $P_\textnormal{p} \kron P_\textnormal{e}$  corresponds to the states $X^{(\textnormal{p})}=i$ and $X^{(\textnormal{e})}=j$ \cite{PMW:62}. The statement $(ii)$ ensures the existence of a node $k$ for every pair $(i,j)$ which is reachable by a walk of equal length from $i$ in $P_\textnormal{p}$ and $j$ in $P_\textnormal{e}$. This condition is equivalent to the node $(k,k)$ being reachable from the pair $(i,j)$ on the Kronecker product of the two Markov chains \cite[Proposition 1]{FH-CATJ:66}.

  Next, we show $(iii) \implies (iv)$. The stochastic matrix $P_\textnormal{p} \kron P_\textnormal{e}$ has a walk from any node $(i,j)$ to some node $(k_1,h_1)$ where $\mathbb{P}[X^{(\textnormal{p})}=k,X^{(\textnormal{e})}=k\,|\, X^{(\textnormal{p})}=k_1,X^{(\textnormal{e})}=h_1]\neq 0$ as there exists a walk $(i,j) \rightarrow (k,k)$. Note that post-multiplying the square matrix $P_\textnormal{p} \kron P_\textnormal{e}$ by $E$ corresponds to setting the columns associated with nodes of the form $(k,k)$ to $\vectorzeros[n^2]$. Thus the row associated with $(k,h)$ has row-sum less than 1. Therefore every node $(i,j)$ has a walk to a node whose row-sum is less than 1 which implies that the matrix $(P_\textnormal{p} \kron P_\textnormal{e}) E$ is convergent by virtue of Lemma~\ref{lem:substochastic} (see Section~\ref{sec:appendix}).

  From this we obtain equation~\refp{eq:Vector2agent}. Since $(iii)$ guarantees the existence of $(I_{n^2}-(P_\textnormal{p}\kron P_\textnormal{e}))^{-1}$, we prove that $(iii) \implies (iv)$.

  Note that the existence of $\Vect{\FMT}$ in $(iv)$ gives $ (iv) \implies (i)$. Thus we have shown that $(i) \implies (ii) \iff (iii) \implies (iv) \implies (i)$. Hence the four conditions are equivalent.
\end{proof}
The above necessary and sufficient conditions give the most general
set of pairs of matrices for which finite meeting times
exist. These conditions are in practice difficult to use for designing transition matrices. Hence, we introduce a few sets of pairs of matrices for which the meeting times are guaranteed to be finite.

\subsection{Sufficient conditions for finiteness}
Consider the following sets of pairs of matrices:
\begin{description}

\item[$\subscr{\PP}{finite}$: finite meeting times.]
  Let $\subscr{\PP}{finite}$ be the set of pairs of transition
  matrices $P_\textnormal{p},P_\textnormal{e}$ satisfying the conditions stated in
  Theorem~\ref{thm:necessary_sufficient_finite} and therefore having
  finite meeting times.
  
\item[$\subscr{\PP}{all-overlap}$: Markov chains with
  all-to-all overlapping absorbing classes.]
  ~\newline Let $\subscr{\PP}{all-overlap}$ be
  the set of pairs of transition matrices $P_\textnormal{p},P_\textnormal{e}$ with the
  following property: $P_\textnormal{p}$ has multiple absorbing classes
  $A^{(\textnormal{p})}_1,A^{(\textnormal{p})}_2,\dots,A^{(\textnormal{p})}_q$ with associated periods
  $d^{(\textnormal{p})}_1,d^{(\textnormal{p})}_2,\dots,d^{(\textnormal{p})}_q$, and $P_\textnormal{e}$ has multiple
  absorbing classes $A^{(\textnormal{e})}_1,A^{(\textnormal{e})}_2,\dots,A^{(\textnormal{e})}_r$ with
  associated periods $d^{(\textnormal{e})}_1,d^{(\textnormal{e})}_2,\dots,d^{(\textnormal{e})}_r$, and for
  each $q' \in \until{q}$ and $r' \in \until{r}$, $A^{(\textnormal{p})}_{q'} \cap
  A^{(\textnormal{e})}_{r'} \neq \phi$ and $\gcd(d^{(\textnormal{p})}_{q'},d^{(\textnormal{e})}_{r'})=1$.

\item[$\subscr{\PP}{SA-overlap}$: single absorbing Markov chains with
  overlapping absorbing]~\newline \textbf{classes.} Let $\subscr{\PP}{SA-overlap}$ be the set of
  pairs of transition matrices $P_\textnormal{p},P_\textnormal{e}$ with the following
  property: $P_\textnormal{p}$ has a single absorbing class $A^{(\textnormal{p})}$ with
  period $d^{(\textnormal{p})}$, and $P_\textnormal{e}$ has a single absorbing class
  $A^{(\textnormal{e})}$ with period $d^{(\textnormal{e})}$, and $A^{(\textnormal{p})} \cap A^{(\textnormal{e})}\neq \phi$
  and $\text{gcd}(d^{(\textnormal{p})},d^{(\textnormal{e})})=1$.

\item[$\subscr{\PP}{one-ergodic}$: one ergodic Markov chain.]
  Let $\subscr{\PP}{one-ergodic}$ be the set of pairs of transition
  matrices $P_\textnormal{p},P_\textnormal{e}$ such that one of the matrices $P_\textnormal{p}$
  or $P_\textnormal{e}$ is ergodic.

\end{description}
Given the above definitions the following theorem holds.
\begin{theorem}[Sufficient conditions for finite meeting times]\label{thm:sufficient_finite}
  The sets of pairs of transition matrices
  $\subscr{\PP}{finite}$, $\subscr{\PP}{all-overlap}$,
  $\subscr{\PP}{SA-overlap}$, $\subscr{\PP}{one-ergodic}$ satisfy
  \begin{equation*}
    ( \subscr{\PP}{one-ergodic} \union \subscr{\PP}{SA-overlap} ) \subset
    \subscr{\PP}{all-overlap} \subset  \subscr{\PP}{finite}.
  \end{equation*}
\end{theorem}

\begin{figure}[h]
  \begin{center}
    \includegraphics[width=.9999\textwidth]{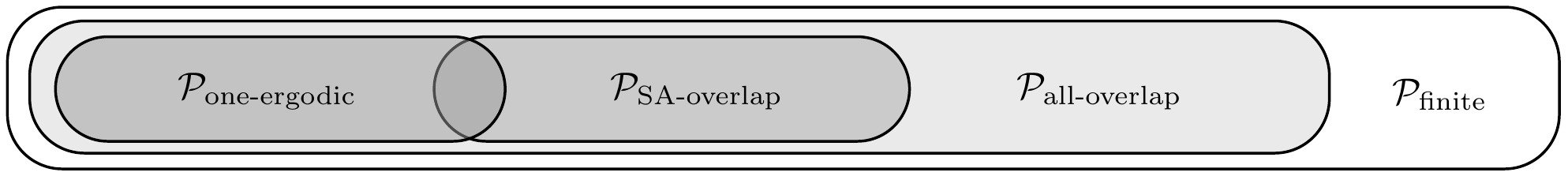}
  \end{center}
  \caption{Sets of pairs of transition matrices with finite meeting times.}
  \label{fig:diagram-sufficient-conditions}
\end{figure}

\begin{proof}
  Before we prove the statement in the theorem we prove a minor result. Consider two Markov chains, each with transition matrices $P_\textnormal{p},P_\textnormal{e} \in \mathbb{R}^{n \times n}$  defined on a digraph $\G$ with nodeset $V = \until{n}$. Let the absorbing classes of $P_\textnormal{p}$ be $A^{(\textnormal{p})}_1,A^{(\textnormal{p})}_2,\dots,A^{(\textnormal{p})}_q$ with periods $d^{(\textnormal{p})}_1,d^{(\textnormal{p})}_2,\dots,d^{(\textnormal{p})}_q$ respectively, and let the absorbing classes of $P_\textnormal{e}$ be $A^{(\textnormal{e})}_1,A^{(\textnormal{e})}_2,\dots,A^{(\textnormal{e})}_r$ with periods $d^{(\textnormal{e})}_1,d^{(\textnormal{e})}_2,\dots,d^{(\textnormal{e})}_r$ respectively. If there exists an absorbing class $A^{(\textnormal{p})}_{q'}$ in $P_\textnormal{p}$ and $A^{(\textnormal{e})}_{r'}$ in $P_\textnormal{e}$ such that $A^{(\textnormal{p})}_{q'} \cap A^{(\textnormal{e})}_{r'} \neq \phi$ and $\text{gcd}(d^{(\textnormal{p})}_{q'}, d^{(\textnormal{e})}_{r'}) = 1$, then there exists a walk from all pairs $(i,j)$, where $i$ is any node from which there exists a walk to $A^{(\textnormal{p})}_{q'}$ and $j$ is any node from which there exists a walk to $A^{(\textnormal{e})}_{r'}$, to a node $(k,k)$ in the digraph associated with the transition matrix $P_\textnormal{p} \kron P_\textnormal{e}$.

  The proof of this result is as follows. Since $A^{(\textnormal{p})}_{q'} \cap A^{(\textnormal{e})}_{r'} \neq \phi$ there exists at least one node $k$ which is accessible from both $i$ and $j$. Since $k$ belongs to the absorbing class $A^{(\textnormal{p})}_{q'}$, starting from the node $i$ there exists all walks of length $u_1 d^{(\textnormal{p})}_{q'} + v_1$ to the node $k$ for all $u_1\geq U_1$, for some $U_1 \in \mathbb{N}$ sufficiently large and some $v_1 \in \mathbb{N}$ such that $0\leq v_1 \leq d^{(\textnormal{p})}_{q'}$. Similarly, since $k$ also belongs to the absorbing class $A^{(\textnormal{e})}_{r'}$, starting from the node $j$ there exists all walks of length $u_2 d^{(\textnormal{e})}_{r'} + v_2$ to the node $k$ for all $u_2\geq U_2$, for some $U_2 \in \mathbb{N}$ sufficiently large and some $v_2\in \mathbb{N}$ such that $0\leq v_2 \leq d^{(\textnormal{e})}_{r'}$. Since $\text{gcd}(d_{q'}^{(\textnormal{p})}, d_{r'}^{(\textnormal{e})}) =1$ we can always find $u_1$ and $u_2$ such that $u_1 d^{(\textnormal{p})}_{q'} + v_1 = u_2 d^{(\textnormal{p})}_{r'} + v_2$. Thus there exists a walk of equal length to the node $k$ from both $i$ and $j$ which ensures that $(k,k)$ is accessible from $(i,j)$.

To prove $\subscr{\PP}{all-overlap} \subset \subscr{\PP}{finite}$ we utilize statement $(iii)$ in Theorem~\ref{thm:necessary_sufficient_finite} to show that for every pair of nodes $(i,j)$, where $i, j$ are nodes in the Markov chain associated with $P_\textnormal{p}, P_\textnormal{e}$, there must exist a walk to a common node of the form $(k,k)$.  Consider a pair of Markov chains $(P_\textnormal{p}, P_\textnormal{e}) \in \subscr{\PP}{all-overlap}$. The states of the Markov chain associated with the transition matrix $P_\textnormal{p}$ can be split into a set of absorbing classes $A^{(\textnormal{p})}_1,A^{(\textnormal{p})}_2,\dots,A^{(\textnormal{p})}_q$ and transient classes $T^{(\textnormal{p})}_1,T^{(\textnormal{p})}_2,\dots,T^{(\textnormal{p})}_s$. Similarly for the Markov chain associated with $P_\textnormal{e}$, the states can be split into a set of absorbing classes $A^{(\textnormal{e})}_1,A^{(\textnormal{e})}_2,\dots,A^{(\textnormal{e})}_r$ and transient classes $T^{(\textnormal{e})}_1,T^{(\textnormal{e})}_2,\dots ,T^{(\textnormal{e})}_t$. We begin by first proving the case for pairs of states belonging to (1) the absorbing classes of both chains, (2) the transient classes of both chains, and finally, (3) transient states of one chain paired with absorbing classes from the other chain.

  Now we will use this result to prove $\subscr{\PP}{all-overlap} \subset \subscr{\PP}{finite}$. We shall show that for pairs of matrices belonging to $\subscr{\PP}{all-overlap}$ the meeting times are finite by concluding that statement $(iii)$ of Theorem~\ref{thm:necessary_sufficient_finite} is satisfied. Consider a pair of Markov chains $(P_\textnormal{p}, P_\textnormal{e}) \in \subscr{\PP}{all-overlap}$. The states of the Markov chain associated with the transition matrix $P_\textnormal{p}$ can be split into a set of absorbing classes $A^{(\textnormal{p})}_1,A^{(\textnormal{p})}_2,\dots,A^{(\textnormal{p})}_q$ and transient classes $T^{(\textnormal{p})}_1,T^{(\textnormal{p})}_2,\dots,T^{(\textnormal{p})}_s$. Similarly for the Markov chain associated with $P_\textnormal{e}$, the states can be split into a set of absorbing classes $A^{(\textnormal{e})}_1,A^{(\textnormal{e})}_2,\dots,A^{(\textnormal{e})}_r$ and transient classes $T^{(\textnormal{e})}_1,T^{(\textnormal{e})}_2,\dots ,T^{(\textnormal{e})}_t$. For statement $(iii)$ in Theorem~\ref{thm:necessary_sufficient_finite} to be satisfied, for every pair of nodes $(i,j)$, where $i, j$ are nodes in the Markov chain associated with $P_\textnormal{p}, P_\textnormal{e}$, there must exist a walk to a node of the form $(k,k)$. To do so we shall initially consider pairs of states belonging to the absorbing classes of both chains, followed by the transient classes of both chains, and finally, transient states of one chain paired with absorbing classes from the other chain, and show that in each case we show a common node exists to which there is a walk of equal length using the proven result.
  
    First, consider nodes $(i,j)$ such that $i$ belongs to an absorbing class $A^{(\textnormal{p})}_{q'}$ where $q' \in \until{q}$ and $j$ belongs to an absorbing class $A^{(\textnormal{e})}_{r'}$ where $r' \in \until{r}$. By definition, every node in an absorbing class has walks to every other node in its class. $\subscr{\PP}{all-overlap}$ gives that $A^{(\textnormal{p})}_{q'} \cap A^{(\textnormal{e})}_{r'} \neq \phi$ and $\text{gcd}(d^{(\textnormal{p})}_{q'}, d^{(\textnormal{e})}_{r'})=1$. Hence the provisions for the result are satisfied for nodes $(i,j) \in A^{(\textnormal{p})}_{q'}\times A^{(\textnormal{e})}_{r'}$ as $A^{(\textnormal{p})}_{q'} \cap A^{(\textnormal{e})}_{r'} \neq \phi$ and $\text{gcd}(d^{(\textnormal{p})}_{q'}, d^{(\textnormal{e})}_{r'})=1$ for every $q' \in \until{q}$ and every $r' \in \until{r}$.
  
    Second, consider nodes $(i,j)$ such that $i$ belongs to a transient class $T^{(\textnormal{p})}_{s'}$ where $s' \in \until{s}$ and $j$ belongs to a transient class $T^{(\textnormal{e})}_{t'}$ where $t' \in \until{t}$.  Since $i$ belongs to a transient class, there must exist a walk to one of the absorbing classes, say $A^{(\textnormal{p})}_{q'}$. Similarly, since $j$ belongs to a transient class, there must exist a walk to one of the absorbing classes, say $A^{(\textnormal{e})}_{r'}$. Hence by the proven result, for each node $(i,j) \in T^{(\textnormal{p})}_{s'} \times T^{(\textnormal{e})}_{t'}$ for every $s' \in \until{s}$ and $t' \in \until{t}$ there exists walks to a node of the form $(k,k)$.

  Finally, consider nodes $(i,j)$ such that $i$ belongs to a transient class $T^{(\textnormal{p})}_{s'}$ and $j$ belongs to an absorbing class $A^{(\textnormal{e})}_{r'}$. Since $i$ belongs to a transient class, there must exist a walk starting from $i$ to an absorbing class, say $A^{(\textnormal{p})}_{q'}$. Thus once again we can apply the earlier stated result for nodes $(i,j) \in T^{(\textnormal{p})}_{s'} \times A^{(\textnormal{e})}_{r'}$ for every $s' \in \until{s}$ and $r' \in \until{r}$. Similarly, the case of nodes belong to absorbing classes in $P_\textnormal{p}$ and transient classes in $P_\textnormal{e}$ also follows. 

  Thus we have exhausted all pairs $(i,j)$ in $P_\textnormal{p} \kron P_\textnormal{e}$ and for each pair found a node of the form $(k,k)$. Therefore $P_\textnormal{p}$ and $P_\textnormal{e}$ satisfy the conditions stated in statement $(iii)$ of Theorem~\ref{thm:necessary_sufficient_finite}, hence guaranteeing finite meeting times and proving that $\subscr{\PP}{all-overlap} \subseteq S$. To show that $\subscr{\PP}{all-overlap} \neq S$, we present a counter-example in Figure~\ref{fig:InfinitemeetingTime}. This concludes the proof for $\subscr{\PP}{all-overlap}\subset \subscr{\PP}{finite}$.

  \begin{figure}[t]
    \centering
    \includegraphics[width=.5\textwidth]{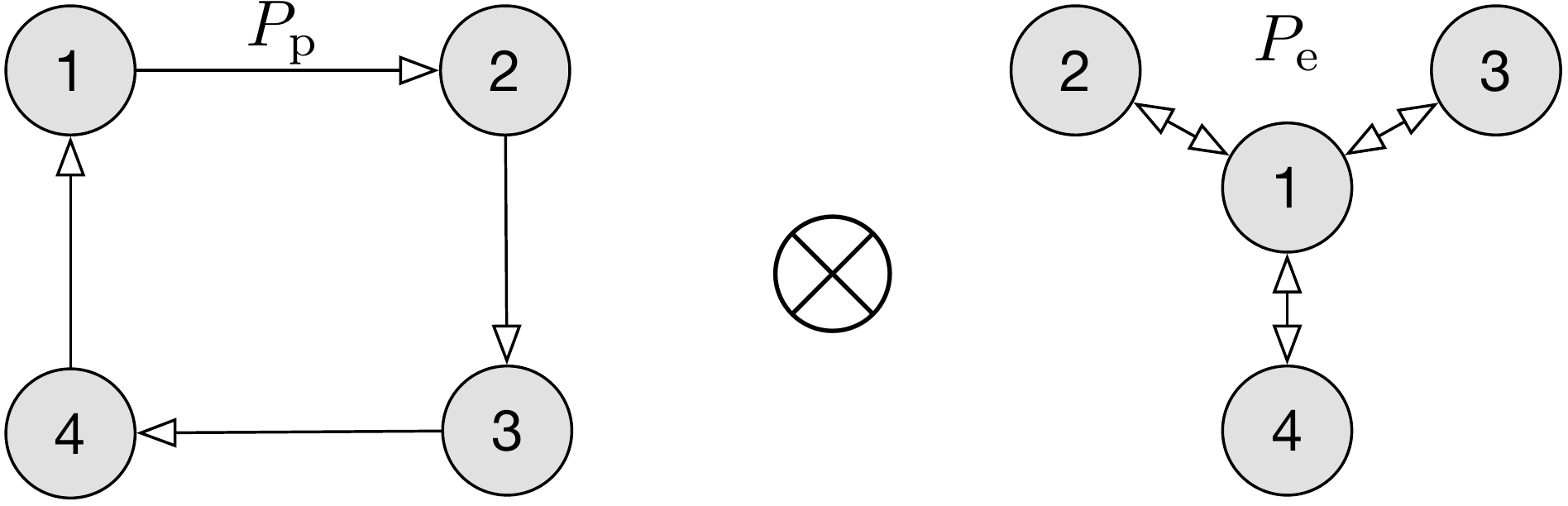}
    \label{fig:InfinitemeetingTime}
    \caption{The periods associated with the pair of Markov Chains shown here are not co-prime; $P_\textnormal{p}$ is a period 4 chain and $P_\textnormal{e}$ is a period 2 chain. However, the meeting times are finite as they satisfy conditions in Theorem~\ref{thm:necessary_sufficient_finite}.}
  \end{figure}

  Now, we prove $\subscr{\PP}{SA-overlap} \subset \subscr{\PP}{all-overlap}$. The pairs of matrices $(P_\textnormal{p},P_\textnormal{e})$ $\in$\newline $\subscr{\PP}{SA-overlap}$ is obtained by considering the subset of matrices which only have a single absorbing class. Thus $\subscr{\PP}{SA-overlap} \subset \subscr{\PP}{all-overlap}$.

  Finally, to prove $\subscr{\PP}{one-ergodic} \subset \subscr{\PP}{all-overlap}$ let us assume without loss of generality that $P_\textnormal{p}$ is irreducible and aperiodic. This would imply that the entire nodeset $V$ is an absorbing state and $d^{(\textnormal{p})}=1$. One can see that $P_\textnormal{p}$ paired with any other matrix $P_\textnormal{e}$ belongs to $\subscr{\PP}{all-overlap}$. Thus $\subscr{\PP}{one-ergodic} \subset \subscr{\PP}{all-overlap}$.
\end{proof}

\subsection{Mean meeting time and relation to hitting times}
 Before we define the mean meeting time for two random walkers, we introduce a minor result.
\begin{remark}\label{rem:pairwise-meeting-times}
  Consider two random walkers moving with transition matrices $P_\textnormal{p},P_\textnormal{e}$ starting from nodes $i,j$ respectively, then the meeting time
  \begin{equation}
    \fmt_{i,j} = (\vect{e}_1 \kron \vect{e}_2)(I_{n^2} - (P_\textnormal{p} \kron P_\textnormal{e})E)^{-1}\vectorones[n^2].
  \end{equation}
\end{remark}
Note that the expression above is a direct result of equation~\refp{eq:Vector2agent}. 

We are now in a position to define the \emph{mean meeting time} of two random walkers. Stationary distributions are well-defined for both $P_\textnormal{p}$ and $P_\textnormal{e}$ when each transition matrix has a single absorbing class. Further the meeting times for matrices with this property are finite only if the absorbing classes overlap and the periods are co-prime as is the case for pairs of transition matrices in $\subscr{\mathcal{P}}{SA-overlap}$. Hence we have the following result.
\begin{corollary}[Mean meeting time]
  Consider two transition matrices $P_\textnormal{p},P_\textnormal{e}$ with stationary distributions $\pi_\textnormal{p},\pi_\textnormal{e}$. The mean meeting time
  \begin{equation}\label{eq:MFMT}
    \MFMT(P_\textnormal{p},P_\textnormal{e}) = \pi_\textnormal{p}^\top \FMT \pi_\textnormal{e} =(\pi_\textnormal{p} \kron \pi_\textnormal{e})^\top \Vect{\FMT},
  \end{equation}
  where $M$ is the matrix of meeting times, is finite if the pair of transition matrices $(P_\textnormal{p},P_\textnormal{e}) \in \subscr{\mathcal{P}}{SA-overlap}$ .
\end{corollary}
\begin{proof}
The mean meeting time can be obtained from the meeting times as
  \begin{align*}
      \MFMT(P_\textnormal{p},P_\textnormal{e}) &= \sum_{i}\sum_{j}\pi_\textnormal{p}^{(i)}\pi_\textnormal{e}^{(j)} m_{i,j}  \\
      &= \sum_{i}\sum_{j}(\pi_\textnormal{p}^{(i)}\vect{e}_{i}\kron \pi_\textnormal{e}^{(j)}\vect{e}_{j})(I_{n^2} - (I_n \kron P)E)^{-1}\vectorones[n^2]  \\
      &= (\pi_\textnormal{p} \kron \pi_\textnormal{e})(I_{n^2} - (P_\textnormal{p} \kron P_\textnormal{e})E)^{-1}\vectorones[n^2]. 
  \end{align*}
  \end{proof}
Further as the following result shows, the hitting times of a Markov chain are equal to the meeting times for the case of a mobile pursuer and stationary evader.

\begin{corollary}[Connection to hitting times and meeting times with stationary evader]
  Consider a stationary evader with distribution $\pi_e$ and a pursuer with an irreducible transition matrix $P_p$ and stationary distribution $\pi_p$, then the following properties hold:
  \begin{enumerate}
  \item  the meeting times between the stationary evader and the pursuer are equal to the pairwise hitting times of $P_p$ and are given by
    \begin{equation}
      h_{i,j} = m_{i,j} = (\vect{e}_1 \kron \vect{e}_2)^\top(I_{n^2} - (I_n \kron P_p)E)^{-1}\vectorones[n^2],
    \end{equation}
    where $h_{i,j}$ is the expected time to travel from node $i$ to node $j$ and
  \item  the mean meeting time between the stationary evader and the pursuer is given by
    \begin{equation}
      \MFMT_\textnormal{stationary}(\pi_e,P_p) = (\pi_e \kron \pi_p)^\top(I_{n^2} - (I_n \kron P_p)E)^{-1}\vectorones[n^2].
    \end{equation}
  \end{enumerate}
\end{corollary}
\begin{proof}
  A stationary evader can be described by the transition matrix $I_n$. However, note that the identity matrix has non-unique stationary distribution hence the evader stationary distribution can be arbitrarily defined given that $\sum_{i=1}^n{\pi^{(i)}_{e}} = 1$. Since $P$ is irreducible, the pair of matrices $(I_n,P)$ belongs to $\mathcal{P}_\text{one-ergodic}$ and hence meeting times are finite. Further the expression for meeting times in this context is identical to that of pairwise hitting times~\cite[Theorem 2.3(i)]{RP-AC-FB:14k}. The mean first meeting time in such a case is
  \begin{align*}
      \MFMT_\textnormal{stationary}(\pi_e,P_p) &= \sum_{i}\sum_{j}\pi_e^{(i)}\pi_p^{(j)} m_{i,j} \\
      &= \sum_{i}\sum_{j}(\pi_e^{(i)}\vect{e}_{i}\kron \pi_p^{(j)}\vect{e}_{j})(I_{n^2} - (I_n \kron P_p)E)^{-1}\vectorones[n^2] \\
      &= (\pi_e \kron \pi_p)(I_{n^2} - (I_n \kron P_p)E)^{-1}\vectorones[n^2].
  \end{align*}
\end{proof}
When the stationary distribution of the evader is equal to the stationary distribution of the pursuer the expression for the meeting time is identical to the mean first passage time of the Markov chain $P_p$~\cite[Theorem 2.3(i)]{RP-AC-FB:14k}.
\subsection{Comparison to existing bounds}
In this section we provide comparisons with existing bounds from literature, a summary of which is presented in Table~\ref{tab:meeting-time-estimates}. We present numerics for a variety of graphs and compare the exact value of the worst meeting time, denoted as $\mathcal{M}_\text{max}$, with bounds on the same quantity from refs.~\cite{CC-RE-HO-TR:13, DC-PR-PW:93} and with the worst hitting time computed using the formula in~\cite{RP-AC-FB:14k} and also a bound on the worst hitting time as described in~\cite{LL:93}.

Most of the bounds discussed here are for random walks i.e., equal probability of transition from a node to every neighbor. The bounds by Aldous~\cite{DJA:91} also hold for all reversible Markov chains. Hence in this section we consider transition matrices only corresponding to random walks. We include self-loops in all transition matrices to ensure aperiodicity. In general meeting times for transition matrices can be significantly smaller than the values discussed here. For example, using transition matrices which are permutation matrices one could obtain $O(n)$ meeting times on all graphs.

\begin{table}
  \label{tab:meeting-time-estimates}
  \centering
    \resizebox{\textwidth}{!}{
  \begin{tabular}{| c|| c | c | c || c | c|}
    \hline
     & \multicolumn{3}{c||}{Meeting time}  & \multicolumn{2}{c|}{Hitting time}\\   \hline
    Quantity & $\mathcal{M}_\text{max}$& \begin{tabular}{@{}c@{}} Bound by \\Cooper et al~\cite{CC-RE-HO-TR:13} \end{tabular}& \begin{tabular}{@{}c@{}} Bound by \\ Coppersmith et al~\cite{DC-PR-PW:93} \end{tabular}&$\mathcal{H}_\text{max}$~\cite{RP-AC-FB:14k}& \begin{tabular}{@{}c@{}}Bound by \\ Lov\'{a}sz~\cite{LL:93} \end{tabular}\\ \hline
    Complexity & ~$O(n^6)$ & $O(n^3)$ & $O(1)$ & $O(n^3)$ & $O(n^3)$ \\ \hline
    Ring & 83.7 & 2488.8 & 856.0 & 150.0 & 2451.8\\ 
    Path & 174.8 & 9249.0 & 856.0 & 551.0 & 17308.6\\ 
    Star & 8.0 & 161.6 & 856.0 & 58.0 & 304.0\\
    Lollipop &  224.0 & 1376.3 & 856.0 & 483.8 & 2107.1 \\
    Lattice & 35.9 & 805.6 & 856.0 & 83.7 & 1233.0\\
    \begin{tabular}{@{}c@{}}Random geometric \\ graph (dense)\end{tabular}  & 22.7 & 342.2 & 856.0  & 92.6 & 1098.8\\
    \begin{tabular}{@{}c@{}}Random geometric \\ graph (sparse)\end{tabular} & 77.0 & 3587.1 & 856.0 & 319.6 & 10138.9 \\
    \hline
  \end{tabular}}
  \caption{Comparison of exact value of worst meeting time with bounds from literature and worst hitting times for random walks on various graphs of size 20 nodes. Values shown for random geometric graphs are averages over 100 instances.}
\end{table}

Bounding the worst meeting time as discussed in~\cite{DJA:91} in terms of the worst pairwise hitting time gives estimates which are of the same order. The computational complexity of exactly obtaining the worst hitting time is $O(n^3)$~\cite[Theorem 2.3(i)]{RP-AC-FB:14k} as compared to $O(n^{6})$ for computing worst meeting times. Thus for small to medium graphs the worst hitting time can be a useful proxy.

The polynomial bound from Coppersmith et al~\cite[Theorem 3]{DC-PR-PW:93} is for sequential motion of the tokens i.e, one of the two tokens moves followed by the other. In order to compare this bound with the expression in equation~\refp{eq:Vector2agent} which is for simultaneous motion of the two random walkers, we divide the bound by two. This bound while easy to compute only provides a maximal estimate of the worst case meeting times.

The bound from Lov\'{a}sz~\cite[Corollary 3.3]{LL:93} is a bound for the worst hitting time. The bounds from Cooper et al~\cite[Theorem 1]{CC-RE-HO-TR:13} and Lov\'{a}sz, both of which involve the spectral gap of the transition matrix, behave similarly in most cases. In general the estimates tend to be one or two orders of magnitude off. The complexity of computing the spectral gap can be cost-effective as this operation can be performed in worst-case $O(n^3)$ and for certain types of matrices in $O(n^2)$.
\section{Multiple pursuers and multiple evaders}\label{sec:multiple_pursuers_evaders}
In this section we extend the results obtained for the single evader and single pursuer case to a group of evaders and a group of pursuers. The mathematical treatment for the finiteness and the closed-form expression of the meeting time for groups follows in similar fashion to the case for the single evader and single pursuer.
\subsection{Finite meeting time between groups} Now consider $L$ pursuers and $M$ evaders. Let $X_t^{(\textnormal{p},1)},X_t^{(\textnormal{p},2)},\dots X_t^{(\textnormal{p},L)} \in\until{n}$  denote the locations of the $L$ pursuers at time $t\in\{0,1,2,\dots\}$. Let $X_t^{(\textnormal{e},1)},X_t^{(\textnormal{e},2)},\dots X_t^{(\textnormal{e},M)} \in\until{n}$  denote the locations of the $M$ evaders at time $t\in\{0,1,2,\dots\}$. For an $L$-tuple of nodes associated with the pursuers $(i_1,i_2,\dots i_L)$ and an $M$-tuple of nodes associated with the evaders $(j_1,j_2,\dots,j_M)$, \emph{the first meeting time among $L$ pursuers and $M$ evaders}, denoted by $T_{i_1i_2\dots i_L,j_1j_2\dots j_M}$, is the first time that one of the pursuers meets one of the evaders. More formally,  $T_{i_1i_2\dots i_L,j_1j_2\dots j_M}$ is
\begin{align*}
  &\min\{t\geq 1 | X_t^{(\textnormal{p},a)} = X_t^{(\textnormal{e},b)}\text{ for some } a \in \until{L} \text{ and } b \in \until{M}\\
  & \text{given that } X_0^{(\textnormal{p},l)} = i_l\enspace\!\!\forall\enspace\! l\in\until{L}\text{ and } X_0^{(\textnormal{e},m)} = j_m\enspace\!\!\forall\enspace\! m\in\until{M}\}.
\end{align*}
Let the transition matrices associated with the $L$ pursuers be $P_\textnormal{p}^{(1)},P_\textnormal{p}^{(2)},\dots,P_\textnormal{p}^{(L)}$ and the transition matrices associated with the $M$ evaders be $P_\textnormal{e}^{(1)},P_\textnormal{e}^{(2)},\dots,P_\textnormal{e}^{(M)}$. The following theorem gives necessary and sufficient conditions for the the first expected meeting time between the $L$ pursuers and $M$ evaders $\fmt_{i_1i_2\dots i_L,j_1j_2\dots j_M}=\mathbb{E}[T_{i_1i_2\dots i_L,j_1j_2\dots j_M}]$. For the sake of brevity, we shall refer to the first expected meeting time in this context as the \emph{group meeting time}.
\begin{theorem}[The group meeting time of multiple Markov chains]\label{thm:finite_multiple}
  Consider Markov chains with transition matrices $P_\textnormal{p}^{(1)},P_\textnormal{p}^{(2)}\dots,P_\textnormal{p}^{(L)},P_\textnormal{e}^{(1)},P_\textnormal{e}^{(2)},\dots,P_\textnormal{e}^{(M)}$ defined on a digraph $\G$ with nodeset $V = \until{n}$.  The following statements are equivalent:
  \begin{enumerate}
  \item for every $i_1,i_2,\dots i_L,j_1,j_2,\dots ,j_M \in\until{n}$, the group meeting time\newline $\fmt_{i_1i_2\dots i_L,j_1j_2\dots j_M}$ is finite,
    
  \item for every $i_1,i_2,\dots,i_L,j_1,j_2,\dots,j_M \in\until{n}$, there exists a node $k$ and a length $\ell$ such that a walk of length $\ell$ exists
    from one of the nodes $i_1,i_2,\dots,i_L$ to $k$ in one of the transition matrices $P_\textnormal{p}^{(1)},P_\textnormal{p}^{(2)},\dots,P_\textnormal{p}^{(L)}$ and a walk of length $\ell$ exists from one of the nodes $j_1,j_2,\dots,j_M$ to $k$ in one of the transition matrices in $P_\textnormal{e}^{(1)},P_\textnormal{e}^{(2)},\dots,P_\textnormal{e}^{(M)}$,

  \item for every $i_1,i_2,\dots,i_L,j_1,j_2,\dots,j_M \in\until{n}$, there exists a
    walk in the digraph associated with the stochastic matrix $P_\textnormal{p}^{(1)} \kron
    P_\textnormal{p}^{(2)}\dots\kron P_\textnormal{p}^{(L)}\kron P_\textnormal{e}^{(1)} \kron P_\textnormal{e}^{(2)} \dots \kron P_\textnormal{e}^{(M)}$ from a node $(i_1,i_2,\dots ,i_L, j_1,j_2,\dots ,j_M)$ to a node of the form $(i^{'}_1,i^{'}_2,\dots,k,\dots,i^{'}_L,j^{'}_1,j^{'}_2,\dots,k,\dots ,j^{'}_M)$, for some $k\in\until{n}$, and

  \item the substochastic matrix $PE$ is convergent and the vector of group meeting times is given by
    \begin{align} \label{eq:VectorLMagent}
      \Vect{\FMT} = (I_{n^{L+M}}-PE)^{-1} \vectorones[n^{L+M}], 
    \end{align} 
    where $\FMT \in \mathbb{R}^{n^L\times n^M}$, $P=P_\textnormal{p}^{(1)} \kron P_\textnormal{p}^{(2)}\dots\kron P_\textnormal{p}^{(L)}\kron P_\textnormal{e}^{(1)} \kron P_\textnormal{e}^{(2)} \dots \kron P_\textnormal{e}^{(M)}$ and $E$ is a binary diagonal matrix with entries $\vectorones[n^{L+M}]-\Vect{[\delta_{i_1i_2\dots i_L,j_1j_2\dots j_M}]}$.
  \end{enumerate}
\end{theorem}
\begin{proof}
    For the nodes $i_1,i_2, \dots , i_L, j_1, j_2, \dots ,j_M $, the group meeting time satisfies the recursive formula
  \begin{equation*}
    \label{eq:t_defMulti}
    T_{i_1 i_2 \dots i_L,j_1 j_2 \dots j_M}\!\! =\!\! 
    \begin{cases}
      \!1, &\!\!\!\!\! \text{w.p.}~\!\!\!\!\displaystyle\sum_{k=1}^n\big(1\!\!-\!\!\prod_{a=1}^L (1-p^{(\textnormal{p},a)}_{i_a,k})\big)\big(1\!\!-\!\!\prod_{b=1}^M (1-p^{(\textnormal{e},b)}_{j_b,k})\big),\\ 
      \!T_{k_1 k_2\dots k_L, h_1 h_2\dots h_M}\!\!+\!1, &\!\!\!\!\! \text{w.p.}~\!\!\!\!\displaystyle\sum_{k_a\neq h_b}\prod_{a=1}^{L} p_{i_a, k_a}^{(\textnormal{p},a)}\displaystyle\prod_{b=1}^{M} p_{j_b, h_b}^{(\textnormal{e},b)}.
    \end{cases}
  \end{equation*}
Note that the symbol $\sum_{k_a \neq h_b}$ is a summation over the indices $k_1,k_2,\dots,k_L,h_1,h_2,\dots,h_M$ such that $k_a \neq h_b$ for every $a\in \until{L}$ and $b\in \until{M}$.
The quantity $(1-\prod_{a=1}^L (1-p^{(\textnormal{p},a)}_{i_a,k}))$ indicates the probability that one of the pursuers will move to node $k$ and $(1-\prod_{b=1}^M (1-p^{(\textnormal{e},b)}_{i_b,k}))$ indicates the probability that one of the evaders  will move to node $k$. Therefore $\sum_{k=1}^n (1-\prod_{a=1}^L (1-p^{(\textnormal{p},a)}_{i_a,k}) )(1-\prod_{b=1}^M (1-p^{(\textnormal{e},b)}_{j_b,k}))$ is the probability that one of the pursuers encounters one of the evaders at a common node.

  Taking the expectation we have
  \begin{equation*}
    \begin{split}
      \mathbb{E}[T_{i_1i_2\dots i_L,j_1j_2\dots j_M}] &=\displaystyle\sum_{k=1}^n (1-\displaystyle\prod_{a=1}^L (1-p^{(\textnormal{p},a)}_{i_a,k}))(1-\displaystyle\prod_{b=1}^M (1-p^{(\textnormal{e},b)}_{j_b,k})) \\ &+  \displaystyle\sum_{k_a\neq h_b}\displaystyle\prod_{a=1}^{L}p_{i_a, k_a}^{(\textnormal{p},a)}\displaystyle\prod_{b=1}^{M}p_{j_b, h_b}^{(\textnormal{e},b)}(\mathbb{E}[T_{k_1k_2\dots k_L , h_1h_2\dots h_M}]+1),\\
      \implies \mathbb{E}[T_{i_1i_2,\dots i_L,j_1j_2 \dots j_M}] & = 1 + \displaystyle\sum_{k_a \neq h_b}\prod_{a=1}^{L}p_{i_a,k_a}^{(\textnormal{p},a)}\prod_{b=1}^{M}p_{j_b, h_b}^{(\textnormal{e},b)}(\mathbb{E}[T_{k_1k_2\dots k_L,h_1h_2\dots h_M}]).
    \end{split}
  \end{equation*}
  Let $\fmt_{i_1i_2\dots i_L, j_1j_2\dots j_M} = \mathbb{E}[T_{i_1i_2\dots i_L, j_1j_2\dots j_M}]$ for every $i_1, i_2, \dots , i_L, j_1, j_2, \dots, j_M \in \until{n}$ and let $\FMT=[\fmt_{i_1i_2\dots i_l,j_1 j_2 \dots j_M}]$. Note that the entries of $\FMT$ can be written as
  \begin{equation*}
    \begin{split}
      \fmt_{i_1i_2\dots i_l,j_1 j_2 \dots j_M} &= 1 + \displaystyle\sum_{k_a\neq h_b}\prod_{a=1}^L p_{j_a, k_a}^{(\textnormal{p},a)} \prod_{b=1}^M p_{j_b, h_b}^{(\textnormal{e},b)} \fmt_{k_1k_2\dots k_l,h_1 h_2 \dots h_M}\\
      &= 1 + \displaystyle\sum^n_{\substack{k_1,k_2,\dots,k_L \\ h_1,h_2,\dots ,h_M}}\displaystyle\prod_{a=1}^L p_{i_a, k_a}^{(\textnormal{p},a)}\prod_{b=1}^M p_{j_b, h_b}^{(\textnormal{e},b)} \fmt_{k_1k_2\dots k_l,h_1 h_2 \dots h_M} -\\
      &\enspace \qquad \displaystyle\sum_{\substack{k_1,k_2,\dots,k_L \\ h_1,h_2,\dots ,h_M}}\delta_{i_1i_2\dots i_L,j_1j_2 \dots j_M}\displaystyle \prod_{a=1}^L  p_{i_a, k_a}^{(\textnormal{p},a)} \prod_{b=1}^M p_{j_b, h_b}^{(\textnormal{e},b)} \fmt_{k_1k_2\dots k_l,h_1 h_2 \dots h_M},
    \end{split}
  \end{equation*}
  where we have rewritten the summation $\sum_{k_a\neq h_b}$ in terms of the generalized kronecker delta function.
  This equation can be written in vector form as
  \begin{equation*}
    \begin{split}
      \Vect{\FMT} = &\vectorones[n^{L+M}]  + (P_\textnormal{p}^{(1)} \kron P_\textnormal{p}^{(2)} \kron \cdots \kron P_\textnormal{p}^{(L)} \kron P_\textnormal{e}^{(1)} \kron P_\textnormal{e}^{(2)} \kron \cdots \kron P_\textnormal{e}^{(M)})\vectorMat[\FMT] \\ &- (P_\textnormal{p}^{(1)} \kron P_\textnormal{p}^{(2)} \kron \cdots \kron P_\textnormal{p}^{(L)} \kron P_\textnormal{e}^{(1)} \kron P_\textnormal{e}^{(2)} \kron \cdots \kron P_\textnormal{e}^{(M)})\\
      &\Vect{[\delta_{i_1i_2\dots i_L,j_1j_2 \dots j_M}]}\vectorMat[\FMT],
    \end{split}
  \end{equation*}
  \begin{equation*}
    \begin{split}
      \implies \Vect{\FMT} = &\vectorones[n^{L+M}]  + P(I_{n^{L+M}}-\Vect{[\delta_{i_1i_2\dots i_L,j_1j_2 \dots j_M}]})\vectorMat[\FMT],\\
      \implies \Vect{\FMT} = &\vectorones[n^{L+M}]  + PE\Vect{\FMT}.\\
    \end{split}
  \end{equation*}
  If the matrix $I_{n^{L+M}}-PE$ is invertible then we have a unique solution to the meeting times. We shall now show that the finiteness of group meeting times as in $(i)$ is equivalent to the existence of walks of equal length to common nodes as mentioned in $(ii)$ and in $(iii)$, which leads to invertibility of $I_{n^{L+M}}-PE$ in $(iv)$.
  
  We start by proving $(i)\implies(ii)$. If we assume that $(i) \centernot \implies (ii)$ then there exists an $L$-tuple of nodes $i_1,i_2,\dots,i_L$ and an $M$-tuple of nodes $j_1,j_2,\dots,j_M$ such that the group meeting time between groups of agents starting from these positions is finite and there exists no walk of equal length to a common node for any possible pursuer-evader pairs. However if there exists no walk of equal length from one of the nodes in $i_1,i_2,\dots,i_L$ and one of the nodes in $j_1,j_2,\dots,j_M$ to a common node, then none of the agents ever meet and the group meeting time is infinite. Hence by contradiction $(i)\implies(ii)$.

Next we show that $(ii)\implies(iii)$. The Kronecker product of the $L$ pursuer transition matrices and the $M$ evader transition matrices gives a joint transition matrix for the agents. The node $(i_1,i_2,\dots,i_L,j_1,j_2,\dots,j_M)$ corresponds to the states $X_\textnormal{p}^{(1)}=i_1,X_\textnormal{p}^{(2)}=i_2,\dots , X_\textnormal{p}^{(L)}=i_L$ and $X_\textnormal{e}^{(1)}=j_1,X_\textnormal{e}^{(2)}=j_2,\dots,X_\textnormal{e}^{(M)}=j_M$. Statement $(ii)$ ensures that there exists a node $k$ to which there is a walk of length $\ell$ from one of the nodes $i_1,i_2,\dots,i_L$ in a pursuer transition matrix $P_\textnormal{p}^{L'}$ and from one of the nodes $j_1,j_2,\dots,j_M$ in an evader transition matrix $P_\textnormal{e}^{M'}$. Starting from any node $i_l$ for any $l \in \until{L}$ there exists a node $i^{'}_l$ such that there exists a walk of length $\ell$ from $i_l$ to $i^{'}_l$ in the transition matrix $P^{(\textnormal{p})}_l$. Similarly starting from $j_m$ there exists some node $j^{'}_m$ to which there exists a walk of length $\ell$ for some $m \in \until{M}$ in the transition matrix $P^{(\textnormal{e})}_m$. Thus there exists walks of length $\ell$ : $i_1 \rightarrow i^{'}_1$, $i_2 \rightarrow i^{'}_2$, $\dots$ , $i_{L'} \rightarrow k$, $\dots$, $i_L \rightarrow i^{'}_{L}$, $j_1 \rightarrow j^{'}_1$, $j_2 \rightarrow j^{'}_2$, $\dots$ , $j_{M'} \rightarrow k$, $\dots$, and $j_M \rightarrow j^{'}_M$. Using Lemma~\ref{lem:kron_walk} (see Section~\ref{sec:appendix}) there exists a walk from $(i_1,i_2,\dots,i_{L'},\dots ,i_L,j_1,j_2,\dots ,j_{M'},\dots ,j_M)$ to a node of the form $(i^{'}_1,i^{'}_2,\dots ,k,\dots ,i^{'}_L,j^{'}_1,j^{'}_2,\dots ,k,\dots ,j^{'}_M)$, thus proving $(ii) \implies (iii)$.

Next, we show $(iii) \implies (iv)$. Note that  post-multiplying the Kronecker product of all transition matrices $P_\textnormal{p}^{(1)} \kron P_\textnormal{p}^{(2)} \kron \dots \kron P_\textnormal{p}^{(L)} \kron P_\textnormal{e}^{(1)} \kron P_\textnormal{e}^{(2)} \kron \dots \kron P_\textnormal{e}^{(M)}$ by $E$ sets columns associated with nodes of the form $(i^{'}_1,i^{'}_2,\dots ,k, \dots, i^{'}_L, j^{'}_1,j^{'}_2, \dots, k, \dots, j^{'}_M)$ to $\vectorzeros[n^{L+M}]$. Therefore every node has a walk to a node whose row-sum is less than 1 which implies that the matrix $PE$ is convergent. From this we obtain equation~\refp{eq:VectorLMagent}. Since $(iii)$ guarantees the existence of $(I_{n^{L+M}} -PE)^{-1}$, we prove that $(iii) \implies (iv)$.

  Note that the existence of $\Vect{M}$ in $(iv)$ gives $(iv) \implies (i)$. Thus we have shown that $(i)\implies (ii) \implies (iii) \implies (iv) \implies (i)$. Hence the four conditions are equivalent.
\end{proof}
The above necessary and sufficient conditions give the most general set of tuples of matrices for which finite meeting times exist. Similar to the single pursuer and single evader case, we present sufficient conditions on the transition matrices which ensure that the meeting times between two groups is finite.
\subsection{Sufficient conditions for finiteness and mean group meeting time}
Consider the following sets of $L+M$-tuples of matrices:
\begin{description}

\item[$\subscr{\PP}{finite}^\text{L,M}$: finite group meeting times.]
  Let $\subscr{\PP}{finite}^\text{L,M}$ be the set of $L+M$-tuples of
transition matrices $P_\textnormal{p}^{(1)}$,
$P_\textnormal{p}^{(2)}$, $\dots$, $P_\textnormal{p}^{(L)}$,
$P_\textnormal{e}^{(1)}$, $P_\textnormal{e}^{(2)}$, $\dots$,
$P_\textnormal{e}^{(M)}$ satisfying the conditions stated in
Theorem~\ref{thm:finite_multiple} and therefore having finite group
meeting times.
\item[$\subscr{\PP}{all-overlap}^\text{L,M}$: Markov chains with
  all-to-all overlapping absorbing classes.]~\newline Let $\subscr{\PP}{all-overlap}^\text{L,M}$ be
  the set of $L+M$-tuples of transition matrices $P_\textnormal{p}^{(1)}$,
$P_\textnormal{p}^{(2)}$, $\dots$, $P_\textnormal{p}^{(L)}$,
$P_\textnormal{e}^{(1)}$, $P_\textnormal{e}^{(2)}$, $\dots$,
$P_\textnormal{e}^{(M)}$ with the following property: for each transition matrix
  $P_\textnormal{p}^{(l)}, l \in \until{L}$ there exists a transition matrix $P_\textnormal{e}^{(m)}$ for some $m \in \until{M}$
  such that $(P_\textnormal{e}^{(l)},P_\textnormal{p}^{(m)}) \in \subscr{\PP}{all-overlap}$.

\item[$\subscr{\PP}{SA-overlap}^\text{L,M}$: single absorbing Markov chains with
  overlapping absorbing]~\newline\textbf{classes.} Let $\subscr{\PP}{SA-overlap}^\text{L,M}$ be
  the set of $L+M$-tuples of transition matrices $P_\textnormal{p}^{(1)}$,
$P_\textnormal{p}^{(2)}$, $\dots$, $P_\textnormal{p}^{(L)}$,
$P_\textnormal{e}^{(1)}$, $P_\textnormal{e}^{(2)}$, $\dots$,
$P_\textnormal{e}^{(M)}$ with the following property: for each transition matrix
  $P_\textnormal{p}^{(l)}, l \in \until{L}$ there exists a transition matrix $P_\textnormal{e}^{(m)}$ for some $m \in \until{M}$
  such that $(P_\textnormal{e}^{(l)},P_\textnormal{p}^{(m)}) \in \subscr{\PP}{SA-overlap}$.

\item[$\subscr{\PP}{one-ergodic}^\text{L,M}$: one ergodic Markov chain.]
  Let $\subscr{\PP}{one-ergodic}^\text{L,M}$ be
  the set of $L+M$-tuples of transition matrices $\multPursuers,\multEvaders$ such that one of the transition matrices is ergodic.
\end{description}
Given the above, we are now in a position to define the \emph{mean group meeting time} of two  sets of random walkers, $L$ pursuers and $M$ evaders. The group meeting times for matrices with single absorbing classes are finite when the $L+M$-tuple $(P_\textnormal{p}^{(1)},P_\textnormal{p}^{(2)},\dots,P_\textnormal{p}^{(L)},P_\textnormal{e}^{(1)},P_\textnormal{e}^{(2)},\dots,P_\textnormal{e}^{(M)}) \in \subscr{\mathcal{P}}{SA-overlap}^\text{L,M}$. Hence, we have the following result.
\begin{corollary}[Mean group meeting time]
  Consider $L+M$ transition matrices $P_\textnormal{p}^{(1)}$, $P_\textnormal{p}^{(2)}$, $\dots$, $P_\textnormal{p}^{(L)}$, $P_\textnormal{e}^{(1)}$, $P_\textnormal{e}^{(2)}$, $\dots$, $P_\textnormal{e}^{(M)}$ with stationary distributions $\pi_\textnormal{p}^{(1)}, \pi_\textnormal{p}^{(2)},\dots,$ $\pi_\textnormal{p}^{(L)},\pi_\textnormal{e}^{(1)},\pi_\textnormal{e}^{(2)},\dots,\pi_\textnormal{e}^{(M)}$. The mean group meeting time
  \begin{equation}
    \MFMT_{L,M} = (\pi_\textnormal{p}^{(1)} \kron \pi_\textnormal{p}^{(2)}\kron \dots \kron \pi_\textnormal{p}^{(L)} \kron \pi_\textnormal{e}^{(1)} \kron \pi_\textnormal{e}^{(2)} \kron \dots \kron \pi_\textnormal{e}^{(M)})^\top \Vect{M},
  \end{equation}
  where $M$ is the matrix of group meeting times, is finite if the $L+M$-tuple of transition matrices $(P_\textnormal{p}^{(1)}$, $P_\textnormal{p}^{(2)}$, $\dots$, $P_\textnormal{p}^{(L)}$, $P_\textnormal{e}^{(1)}$, $P_\textnormal{e}^{(2)}$, $\dots$, $P_\textnormal{e}^{(M)})$ $\in \subscr{\mathcal{P}}{SA-overlap}^\text{L,M}$ .
\end{corollary}

A word on the computational complexity for the multiple pursuer-evader case: since the general expression for the group meeting time among groups of pursuers and evaders involve extensive use of the Kronecker product, the memory and computational resources necessary are significantly affected by the curse of dimensionality. The matrix $\pursuersKronEvaders$ contains $n^{(L+M)}$ elements. Inversion of a full matrix would require $O(k^3)$ operations lending an undesirable complexity of $O(n^{3(L+M)})$~\cite{PB-MC-AS:97}. Most practical solutions to the transition matrices benefit from the sparse nature of the graphs. A sparse system of equations can be solved with complexity $O(nnz)$ where $nnz$ is the number of non-zero elements. For the Kronecker product of $L+M$ transition matrices defined on the same graph the number of non-zero elements is $|E|^{(L+M)}$, and with a sparse solver the group meeting time can be computed in $O(|E|^{(L+M)})$ operations. For further details see~\cite{JRG-CM-RS:92}.

\section{Meeting times for continous-time Markov chains}\label{sec:CTMC_meeting}
In this section we formulate the meeting time between two continuous-time Markov chains. The essence of the proof is to use the fact that the meeting time of two continuous-time Markov chains is equivalent to the hitting time on the joint transition matrix generated by the Kronecker product of the digraphs associated with the two chains.

\subsection{The meeting time of two continuous-time Markov chains}
Consider the pursuer and evader performing random walks on a set of nodes $\V :=\until{n}$  with digraphs $\G_\textnormal{p} =(V,E_\textnormal{p},Q_\textnormal{p})$, $\G_\textnormal{e}=(V,E_\textnormal{e},Q_\textnormal{e})$, edge sets $E_\textnormal{p},E_\textnormal{e}\subset \V \times \V$, and transition rate matrices $Q_\textnormal{p}$, $Q_\textnormal{e}$. Let $P_\textnormal{p}(t),P_\textnormal{e}(t)$ denote the transition matrices of the purser and evader at time $t$. Let $\fmt_{i,j}$ denote the expected first meeting time for a pursuer starting from node $i$ and an evader starting from node $j$, which shall be referred to simply as the meeting time. Then the following theorem holds.
\begin{theorem}[The meeting time of two continuous-time Markov chains]\label{thm:CTMC_single}
  Consider two Markov chains with transition rate  matrices $Q_\textnormal{p}$ and
  $Q_\textnormal{e}$ defined on a digraph $\G$ with nodeset $V = \until{n}$.  The
  following statements are equivalent:
 \begin{enumerate}
  \item for each pair of nodes $i,j$, the expected first meeting time
    $\fmt_{i,j}$ from nodes $i$ and $j$ is finite,
    
  \item for each pair of nodes $i,j$, there exists a node $k$
    such that a walk exists
    from $i$ to $k$ and a walk exists from
    $j$ to $k$,

  \item for each pair of nodes $i,j$, there exists a
    walk in the digraph associated with the transition rate matrix $Q_\textnormal{p} \kron
    I_n + I_n \kron Q_\textnormal{e}$ from $(i,j)$ to a node $(k,k)$, for some
    $k\in\until{n}$, and

  \item the matrix $E(I_{n^2}-(Q_\textnormal{p}\kron I_n + I_n \kron Q_\textnormal{e})) - I_{n^2}$ is invertible and the first meeting times are given by the unique solution to
    \begin{equation}\label{eq:CTMC-fmt}
      \Vect{\FMT} =(E(I_{n^2}-(Q_\textnormal{p}\kron I_n + I_n \kron Q_\textnormal{e})) - I_{n^2})^{-1} E\vectorones[n^2]
    \end{equation}
    where $\FMT\in \mathbb{R}^{n \times n}$ and $E \in \mathbb{R}^{n^2 \times n^2}$ is a binary diagonal matrix with diagonal entries $\vectorones[n^2] -\Vect{I_n}$.
  \end{enumerate}
\end{theorem}
\begin{proof}
 Consider the joint evolution of the two continuous-time Markov chains on the Kronecker product graph given by the Kronecker product $P_\textnormal{p}(t) \kron P_\textnormal{e}(t)$. The transition rate matrix for this Markov chain is easy to derive. Consider
  \begin{align*}
    \frac{d}{dt} (P_\textnormal{p}(t) \kron P_\textnormal{e}(t)) = &P_\textnormal{p}(t)Q_\textnormal{p}\kron P_\textnormal{e}(t) + P_\textnormal{p}(t)\kron P_\textnormal{e}(t)Q_\textnormal{e}\\
    =&(P_\textnormal{p}(t) \kron P_\textnormal{e}(t)) (Q_\textnormal{p} \kron I_n + I_n \kron Q_\textnormal{e}),
  \end{align*}
  where we have used the product rule for derivatives and Lemma~\ref{def:KronProps1} to obtain the joint transition rate matrix as $Q_\textnormal{p} \kron I_n + I_n \kron Q_\textnormal{e}$.

  The $n\times n$ block entries of the joint transition rate matrix
  \begin{align*}
    Q_\textnormal{p} \kron I_n + I_n \kron Q_\textnormal{e} = 
    \left[
    \begin{array}{cccc}
      q^{(\textnormal{p})}_{1,1} I_n+Q_\textnormal{e} & q^{(\textnormal{p})}_{1,2} I_n & \dots & q^{(\textnormal{p})}_{1,n} I_n \\
      q^{(\textnormal{p})}_{2,1} I_n & q^{(\textnormal{p})}_{2,2} I_n + Q_\textnormal{e}& \dots  &q^{(\textnormal{p})}_{2,n} I_n \\
      \vdots & \vdots & \ddots & \vdots\\
      q^{(\textnormal{p})}_{n,1} I_n & q^{(\textnormal{p})}_{n,2}I_n &  \dots& q^{(\textnormal{p})}_{n,n}I_n +Q_\textnormal{e}
    \end{array} \right].
  \end{align*}

  The meeting times for the two transition rate matrices correspond to hitting times from nodes on the joint transition rate matrix $Q_\textnormal{p} \kron I_n + I_n \kron Q_\textnormal{e}$ to the set of common nodes of the form $(k,k)$. The solution to hitting times for continuous-time Markov chains is given in \cite[Theorem 3.3.3]{JRN:97}. We restate the result here for the sake of completeness. Given a transition rate matrix $Q=[q_{a,b}]$ defined on a set of nodes $A$ and a subset $S\subset A$, the expected meeting times starting from a node $a\in A$ to the set $S$ denoted by $h_a^S$ is given by the solution to the system of equations
  \begin{equation} \label{eq:mfhtContinuous}
    \begin{cases}
      h_a^S =0 &\text{for}\enspace a\in S\\
      -\sum_{b\in A} q_{a,b} h_b^S = 1 &\text{for}\enspace a\notin S.
    \end{cases}
  \end{equation}
 The meeting times can be obtained as the solution to the system of equations above with transition rate matrix given by  $Q_\textnormal{p}\kron I_n + I_n \kron Q_\textnormal{e}$ and $S=\{(k,k)\,|\,k\in V\}$. Denoting $Q_\textnormal{p}\kron I_n + I_n \kron Q_\textnormal{e}$ and  $S=\{(k,k)\,|\,k\in V\}$ by $Q^\text{eff}$ and $S_\text{common}$, respectively, the system of equations in (\ref{eq:mfhtContinuous}) can be written as
  \begin{equation*}
    \begin{cases}
      m_{i,j} =0 &\text{for}\enspace (i,j)\in S_\text{common}\\
      -\sum_{k\in V}\sum_{h\in V}  Q^\text{eff}_{(i,j),(k,h)}m_{k,h} = 1 &\text{for}\enspace (i,j)\notin S_\text{common}.
    \end{cases}
  \end{equation*}
  These equations can be re-written in vector form as
  \begin{equation*}
    -EQ^\text{eff}\Vect{M} = E\vectorones[n^2],\qquad (E-I_{n^2})\Vect{M}= \vectorzeros[n^2].
  \end{equation*}
  Adding the above two equations we obtain equation~\refp{eq:CTMC-fmt}. If the matrix $E(I_{n^2}-Q^\text{eff})-I_{n^2}$ is invertible then we have a unique solution to the meeting times. We shall now show that the finiteness of meeting times as in $(i)$ is equivalent to the existence of walks to common nodes as mentioned in $(ii)$ and $(iii)$, which leads to invertibility of $E(I_{n^2}-Q^\text{eff})-I_{n^2}$ in $(iv)$.

   We start by proving that $(i) \implies (ii)$. If we assume that $(i) \centernot \implies (ii)$, then there exists a pair of nodes $i$ and $j$ such that the expected first meeting time is finite and there exists no walk to a common node in $V$. However if there exists no walk to a common node, then the agents never meet and the first meeting time is always infinite. Hence by contradiction $(i) \implies (ii)$.
  
  Now we shall prove that $(ii) \implies (iii)$. Since the matrix $Q_\textnormal{e}$ is added to every diagonal block of the joint transition rate matrix if there exists a walk in the transition rate matrix $Q_\textnormal{e}$ from $j$ to $k$ then there exists a walk in $Q_\textnormal{p}\kron I_n + I_n \kron Q_\textnormal{e}$ from $(i,j)$ to $(i,k)$ for every $i\in\until{n}$. Also note that the off diagonal block elements are of the form $q^{(p)}_{i',j'} I_n$. One can verify that because of this structure if there is a walk from $j \rightarrow k$ in $Q_\textnormal{p}$ then in $Q_\textnormal{p}\kron I_n + I_n \kron Q_\textnormal{e}$ there exists a walk from $(i,j)\rightarrow (k,j)$ for every $j \in \until{n}$. Hence if there is a walk $i \rightarrow k$ in $Q_\textnormal{p}$ and $j \rightarrow k$ in $Q_\textnormal{e}$ then there exists walks $(i,j)\rightarrow (i,k)\rightarrow (k,k)$, thus proving $(ii) \implies (iii)$. 

  Finally we shall prove $(iii)\implies(iv)$. First consider the modified transition rate matrix $E(Q_\textnormal{p} \kron I_n + I_n \kron Q_\textnormal{e})$. The matrix $E$ sets the rows corresponding to nodes of the form $(j,j)$ to $\vectorzeros[n^2]^\top$. The rank of a transition rate matrix is $n-d$ where $d$ is the number of sinks in the transition rate matrix~\cite{DMF-JAJ:75}. The matrix $E(Q_\textnormal{p} \kron I_n + I_n \kron Q_\textnormal{e})$ has at least $n$ sinks corresponding to the elements $(k,k)$ for every $j\in\until{n}$. If every node has a path to a node of the form $(k,k)$  as in $(iii)$, then there are only exactly $n$ sinks. This is becuase there are exactly $n$ nodes of the form $(k,k)$. Thus the rank of  $E(Q_\textnormal{p} \kron I_n + I_n \kron Q_\textnormal{e})$ is $n^2-n$ implying that this matrix has $n$ null eigenvectors. One can verify that the null eigenvectors (and basis vectors for the kernel) are given by $\vect{e}_1,\vect{e}_{n+1},\vect{e}_{2n+2},\dots, \vect{e}_{n^2}$. Let the other eigenvectors be $v_1,v_2,\ldots,v_{n^2-n}$.  Since the kernel of $E(Q_\textnormal{p} \kron I_n + I_n \kron Q_\textnormal{e})$ is spanned by $\vect{e}_1,\vect{e}_{n+1},\vect{e}_{2n+2},\dots, \vect{e}_{n^2}$, the eigenvectors of the same matrix can be uniquely constructed by ensuring they are orthogonal to the kernel, i.e. $v_p^\top \vect{e}_q=0$ for every $p\in \until{n^2-n}$ and $q\in \{1,n+1,\dots,n^2\}$. Let us denote the eigenvalues associated with these eigenvectors as $\lambda_1,\lambda_2,\ldots, \lambda_{n^2-n}$. Consider the matrix $E(Q_\textnormal{p} \kron I_n + I_n \kron Q_\textnormal{e})+ (I_{n^2}-E)$. We shall show that  $E(Q_\textnormal{p} \kron I_n + I_n \kron Q_\textnormal{e})+ (I_{n^2}-E)$ has the same eigenvectors as $E(Q_\textnormal{p} \kron I_n + I_n \kron Q_\textnormal{e})$. It is easy to see that now the eigenvectors $\vect{e}_1,\vect{e}_{n+1},\ldots,\vect{e}_{n^2}$ have eigenvalues $1$. One can verify $v_q$ is still an eigenvector but with eigenvalue $\lambda_q+1$. Note that $E(Q_\textnormal{p} \kron I_n + I_n \kron Q_\textnormal{e})$ is positive semi-definite from Gershgorin's disk theorem~\cite{CDM:01}. Since $E(Q_\textnormal{p} \kron I_n + I_n \kron Q_\textnormal{e})$ has all non-negative eigenvalues we are assured that $E(Q_\textnormal{p} \kron I_n + I_n \kron Q_\textnormal{e})+ (I_{n^2}-E)$ has all positive eigenvalues and is invertible. Thus if $(iii)$ holds $-(E(Q_\textnormal{p} \kron I_n + I_n \kron Q_\textnormal{e})+(I_{n^2}-E))$ has full rank and is invertible. Thus equation~\refp{eq:CTMC-fmt} gives the unique solution to the meeting times. Therefore $(iii)\implies (iv)$.

  Note that the existence of $\Vect{M}$ in $(iv)$ gives $(iv)\implies (i)$. Thus we have shown that $(i) \implies (ii) \implies (iii) \implies (iv) \implies (i)$. Hence the four conditions are equivalent.
\end{proof}
One can derive sets of pairs of transition rate matrices for which meeting times are guaranteed to be finite akin to the discrete-time case: $\mathcal{Q}_\text{all-overlap}$, $\mathcal{Q}_\text{SA-overlap}$ and $\mathcal{Q}_\text{one-ergodic}$. The sets are almost identical in description except for the fact that periodicity conditions are no longer necessary. A notion of mean meeting time is applicable to the set of transition rate matrices belonging to $\mathcal{Q}_\text{SA-overlap}$.
\subsection{The group meeting times of multiple continuous-time Markov chains} The setup for multiple pursuers and evaders following continuous-time Markov chains on a common graph is identical to the multiple pursuers and multiple evaders in the discrete time case. Consider pursuer transition rate matrices $\multPursuersCTMC$ and evader transition rate matrices $\multEvadersCTMC$.
\begin{theorem}[The group meeting time of multiple continuous-time Markov chains]\label{thm:finite_CTMC_multiple}
  Consider Markov chains with transition rate matrices $\multPursuersCTMC,$ $Q_\textnormal{e}^{(1)},$ $Q_\textnormal{e}^{(2)},$ $\dots, Q_\textnormal{e}^{(M)}$ defined on a digraph $\G$ with nodeset $V = \until{n}$.  The following statements are equivalent:
  \begin{enumerate}
  \item for every $i_1,i_2,\dots i_L,j_1,j_2,\dots ,j_M \in\until{n}$, the expected first meeting time $\fmt_{i_1i_2\dots i_L,j_1j_2\dots j_M}$ is finite,
  \item for every $i_1,i_2,\dots,i_L,j_1,j_2,\dots,j_M \in\until{n}$, there exists a node $k$ such that there exists a walk from one of the nodes $i_1,i_2,\dots,i_L$ to $k$ in one of the transition rate matrices $\multPursuersCTMC$ and a walk exists from one of the nodes $j_1,j_2,\dots,j_M$ to $k$ in one of the transition matrices in $Q_\textnormal{e}^{(1)},$ $Q_\textnormal{e}^{(2)},$ $\dots,$ $Q_\textnormal{e}^{(L)}$,
  \item for every $i_1,i_2,\dots,i_L,j_1,j_2,\dots,j_M \in\until{n}$, there exists a
    walk in the digraph associated with the transition rate matrix $\sum_{l=1}^L I_{n^{l-1}}\kron Q_\textnormal{p}^{(l)}\kron I_{n^{L+M-l}}$ $+$ $\sum_{m=1}^M I_{n^{L+m-1}}\kron Q_\textnormal{e}^{(m)} \kron I_{n^{M-m}}$ from a node $(i_1,i_2,\dots ,i_L, j_1,j_2,$ $\dots ,j_M)$ to a node of the form $(i^{'}_1,i^{'}_2,\dots,k,\dots,i^{'}_L,j^{'}_1,j^{'}_2,\dots,k,\dots ,$ $j^{'}_M)$, for some $k\in\until{n}$, and
  \item the matrix $E(I_{n^{L+M}}-Q)-I_{n^{L+M}}$ is invertible and the expected first meeting time is given by
    \begin{align}\label{eq:CTMC_multiple}
      \Vect{\FMT} = (E(I_{n^{L+M}}-Q)-I_{n^{L+M}})^{-1} \vectorones[n^{L+M}], 
    \end{align} 
    where $\FMT \in \mathbb{R}^{n^L\times n^M}$, $\sum_{l=1}^L I_{n^{l-1}}\kron Q_\textnormal{p}^{(l)}\kron I_{n^{L+M-l}}$ $+$ $\sum_{m=1}^M I_{n^{L+m-1}}\kron Q_\textnormal{e}^{(m)}$ $\kron I_{n^{M-m}}$ and $E$ is a binary diagonal matrix with entries $\vectorones[n^{L+M}]-\Vect{[\delta_{i_1i_2\dots i_L,j_1j_2\dots j_M}]}$.
  \end{enumerate}
\end{theorem}
We state this result without proof as it utilizes the same technique as in the proof of Theorem~\ref{thm:CTMC_single}. The proof of this result involves constructing the joint transition rate matrix of all agents on the Kronecker digraph and then computing the hitting time to the set of $(L+M)$ tuples of nodes such that one of the first $L$ entries is the same as one of the next $M$ entries. The complexity of computing meeting times using equation~\refp{eq:CTMC_multiple} for continuous-time Markov chains is $O(n^{3(L+M)})$ as it involves inversion of a matrix which has $n^{L+M}$ elements, which is identical to the discrete-time case.
\section{Conclusions}\label{sec:conclusion}
We have studied the meeting time of multiple random walkers on a graph and have presented necessary and sufficient conditions for finiteness and novel closed-form expressions for the expected time to meeting between a single pursuer and a single evader, multiple pursuers and multiple evaders, and extended the treatment to continuous-time chains. We also provide sufficient conditions for certain pairs (or tuples) of Markov chains that satisfy conditions on their absorbing classes to have finite meeting times. Finally, we discuss connections to other metrics relevant to Markov chains such as the hitting time.

Several future directions of interest are left open by this work. Though we provide closed-form expressions here, the complexity involved in the calculation makes the computation expensive for large number of agents and on graphs with large number of nodes. It would be of practical interest to devise a formulation which has lower complexity. The literature on computationally efficient methods to calculate the SimRank of nodes on a graph might provide alternative formulations~\cite{CL-JH-GH-XJ-YS-YY-TW:10}. An extension of the work discussed here would be to consider walkers moving with travel times similar to the case of doubly weighted graphs described in~\cite{RP-PA-FB:14b}. Finally, it would be interesting to see if the notion of walks on Kronecker product graphs, which is extensively used in this paper, can be utilized to provide closed-form formulations of other quantities such as the expected time to capture of all evaders.
\section{Appendix}\label{sec:appendix}
For completeness we include the following lemmas which are necessary for the proof of Theorem~\ref{thm:necessary_sufficient_finite} and Theorem~\ref{thm:finite_multiple}.
\begin{lemma} [Convergence of substochastic matrices]  \label{lem:substochastic}
  Let $P \in \mathbb{R}^{n \times n}$ be a substochastic matrix with at least one row-sum $\sum_{j=1}^n P_{i,j} <1$. If for every node there exists a walk to a node with row-sum less than 1, then $P$ is convergent.
\end{lemma}
\begin{lemma} [Existence of walks on Kronecker products]  \label{lem:kron_walk}
  Let $P_1,P_2,\dots,P_N \in \mathbb{R}^{n \times n}$ be stochastic matrices. If there exists a walk from $i_1\rightarrow j_1$ in $P_1$, $i_2\rightarrow j_2$ in $P_2$, $\dots$, and $i_N\rightarrow j_N$ in $P_N$ of equal length, then there exists a walk from $(i_1,i_2,\dots,i_N)$ to $(j_1,j_2,\dots,j_N)$ in $P_1\kron P_2 \kron \dots \kron P_N$.
\end{lemma}
\bibliographystyle{plain}
\bibliography{alias,Main,FB}
\end{document}